\documentclass[12pt]{article}

\usepackage{arydshln}
\usepackage{mathrsfs}
\usepackage{amsmath}
\usepackage{amsfonts}
\usepackage{amssymb}
\usepackage[all,cmtip]{xy}
\usepackage[colorlinks]{hyperref}
\usepackage{color}
\usepackage{cleveref}
\usepackage{ulem}
\usepackage[abs]{overpic}
\usepackage{float}
\usepackage{enumitem}

\usepackage{amsmath, amsthm, amsfonts, amssymb}
\usepackage{etoolbox}
\patchcmd{\thebibliography}
  {\settowidth}
  {\setlength{\itemsep}{0pt plus 0.1pt}\settowidth}
  {}{}
\apptocmd{\thebibliography}
  {\small}
  {}{}

\def\<{\langle}
\def\>{\rangle}

\def\ra{\rightarrow}
\def\p{\partial}
\def\a{\alpha}
\def\wh{\widehat}
\def\wt{\widetilde}

\def\D{{\cal D}}
\def\O{\Omega}

\def \sm{\setminus}
\def\-{\overline}

\def\o{\omega}
\def\ov{\overline}
\def\e{\epsilon}
\def\D{\Delta}

\def\L{\Lambda}
\def\h{\hbox}
\def\d{\delta}

\def\d{\delta}
\def\d{\delta}
\def\b{\beta}
\def\a{\alpha}

\def\RR{{\mathbb R}}
\def\CC{{\mathbb C}}

\def\NN{{\mathbb N}}

\def\BB{{\mathbb B}}

\def\-{\overline}

\def\ld{\lambda}
\def\O{\Omega}
\def\o{\omega}

\def\D{\Delta}

\def\sm{\setminus}
\def\L{\Lambda}

\def\h{\hbox}

\def\wt{\widetilde}
\def\ra{\rightarrow}
\def\p{\partial}

\def\d{\delta}

\def\a{\alpha}
\def\d{\delta}
\def\E{\mathcal E}

\def\D{\Delta}
\def\d{\delta}
\def\b{\beta}

\def\a{\alpha}

\def\a{\alpha}

\def\CC{\mathbb C}

\def\RR{{\mathbb R}}

\def\BB{{\mathbb B}}

\def\NN{{\mathbb N}}

\def\ra{\rightarrow}
\def\p{\partial}
\def\wt{\widetilde}
\def\-{\overline}

\newtheorem{theorem}{Theorem}[section]
\newtheorem{lemma}[theorem]{Lemma}

\newtheorem{proposition}[theorem]{Proposition}

\newtheorem{definition}[theorem]{Definition}
\newtheorem{example}[theorem]{Example}
\newtheorem{remark}[theorem]{Remark}

\newtheorem{conjecture}[theorem]{Conjecture}

\newcommand{\Addresses}{{
		\bigskip
		\footnotesize
  
         Hanlong Fang, \par\nopagebreak
        \textsc{School of Mathematical Sciences, Peking University, Beijing, 100871, China.}\par\nopagebreak
         \textit{E-mail address:} \href{mailto:hlfang@pku.edu.cn)}{hlfang@pku.edu.cn}

		\medskip
		
	    Xiaojun Huang, \par\nopagebreak
	   \textsc{Department of Mathematics, Rutgers University, New Brunswick, NJ 08903, USA.}\par\nopagebreak
		\textit{E-mail address}: \href{mailto:huangx$@$math.rutgers.edu}{huangx$@$math.rutgers.edu}

		\medskip
		
	    Wanke Yin, \par\nopagebreak
	   \textsc{School of Mathematics and Statistics, Wuhan University, Wuhan, Hubei, 430072, China }\par\nopagebreak
		\textit{E-mail address}: \href{mailto:wankeyin@whu.edu.cn}{wankeyin@whu.edu.cn}
		
		\medskip
		
	    Zhengyi Zhou, \par\nopagebreak
	   \textsc{Morningside Center of Mathematics, Chinese Academy of Sciences (CAS)}, \par\nopagebreak
         \textsc{Academy of Mathematics and Systems Science, CAS, Beijing, 100190, China}\par\nopagebreak
		\textit{E-mail address}: \href{mailto:zhyzhou@amss.ac.cn}{zhyzhou@amss.ac.cn}

}}

\begin{document}
\medskip
\title{\bf Bounding smooth Levi-flat hypersurfaces  in  a Stein manifold}

\author {Hanlong Fang\footnote{Supported by National Key R\&D Program of China under Grant No.2022YFA1006700 and NSFC-12201012}, Xiaojun Huang\footnote{Supported in part by  DMS-2247151}, Wanke Yin\footnote{Supported in part by  NSFC-12171372} and Zhengyi Zhou\footnote{Supported by National Key R\&D Program of China under Grant No.2023YFA1010500, NSFC-12288201 and NSFC-12231010}}
\date{}

\maketitle
\begin{abstract}
This paper is concerned with the problem of constructing a smooth Levi-flat hypersurface locally or globally attached to a real codimension two submanifold in $\CC^{n+1}$, or more generally in a Stein manifold, with elliptic CR singularities, a research direction originated from a fundamental and classical paper of E. Bishop. 
Earlier works along these lines include those by many prominent mathematicians working both on complex analysis and geometry. 
We prove that a compact smooth (or, real analytic) real codimension two submanifold $M$, that is contained in the boundary of a smoothly bounded strongly pseudoconvex domain, with a natural and necessary condition called CR non-minimal condition at CR points and with two elliptic CR singular points bounds a smooth-up-to-boundary  (real analytic-up-to-boundary, respectively) Levi-flat hypersurface $\wh{M}$. This answers a well-known question left open from the work of Dolbeault-Tomassini-Zaitsev,
or a generalized version of a problem already asked by Bishop in 1965.  Our study here reveals an intricate interaction of several complex analysis with other fields such as symplectic geometry and foliation theory. 


\end{abstract}
\section{Introduction} Let $M\subset {\mathcal M}$ be  a smooth real submanifold of a Stein manifold $\mathcal M$ of complex dimension $(n+1)$. For
any point $q\in M$, write $T^{(1,0)}_qM:=\mathbb CT_qM\cap T^{(1,0)}_q{\mathcal M}$ for the  tangent  subspace of type $(1,0)$ of $M$ at $q$. The complex dimension  of  $T^{(1,0)}_qM$, denoted by  $\h{dim}_{CR}(q)$, is the simplest holomorphic invariant  of  the germ of $M$ at
$q$. $\h{dim}_{CR}(q)$ is an upper semi-continuous function. When $\h{dim}_{CR}(q)$ is
constant for $q\in M$ near $p$, we call $p$ a CR point of $M$ with  CR dimension $\h{dim}_{CR}(q)$.
Otherwise, $p$ is called a CR singular point. The study of 
geometric, analytic, and dynamical  properties for $M$ near a CR singular point
has attracted tremendous attention in the subject of several complex variables and related subjects such as symplectic geometry \cite{Eli89, Gromov,Hofer}
since the celebrated paper of
Bishop in 1965 \cite{Bis}.

Bishop considered the case when $M$ is a real
surface  in ${\mathbb C}^{2}$ with a CR singular point at
$p$  or more generally a $(n+1)$-manifold in $\CC^{n+1}$. He discovered that under a certain
natural non-degeneracy assumption and a certain holomorphically
invariant convexity of $M$ near $p$,  now called ellipticity,  $M$ has a non-trivial local
hull of holomorphy $\widehat{M}$ and has a very rich holomorphically
invariant geometric structure. Bishop conjectured that $\widehat{M}$
is a Levi-flat submanifold which has more or less the same
regularity as $M$ does even up to $M$ near $p$. Bishop's problem was
confirmed in a sequence of papers by Kenig-Webster \cite{KW1,KW2}, Bedford-Gaveau \cite{BG},
Moser-Webster \cite{MW}, Moser \cite{Mos}, Huang-Krantz \cite{HK}, and finally in Huang \cite{Hu1}.
Bishop also asked a global bounding problem:  When does a compact real surface in $M\subset \CC^2$ with exactly two elliptic complex tangents bound a Levi-flat hypersurface in $\CC^2$? Under   a necessary  convexity assumption   that $M$ is contained in the boundary of a  smoothly bounded pseudoconvex domain,  this global problem of Bishop was  more or less  answered    in the work of Bedford-Gaveau \cite{BG} and Bedford-Klingenberg \cite{BK}. 
More  works on the geometric, analytic and dynamical properties near CR singular points  at least include the papers by Forstneri\v{c} \cite{For}, Gong \cite{Gon1,Gon2},
Gong-Lebl \cite{GL}, Gong-Stolotvich \cite{GS1,GS2}, Lebl \cite{Leb},  Burcea
\cite{Bur1,Bur2}, Baounendi-Ebenfelt-Rothschild \cite{BER}, Huang-Yin \cite{HY1,HY2,HY3,HY4}, Lebl-Noell-Ravisankar
\cite{LNR1},  Gupta-Wawrzyniak \cite{GW}, Stolovitch-Zhao \cite{SZ}, Klime\v{s}-Stolovitch \cite{KS},  and many references therein.

In an important development,  Dolbeault-Tomassini-Zaitsev in 2005 \cite{DTZ0, DTZ1} took up the generalized Bishop problem for a real codimension two
submanifold $M\subset {\mathbb C}^{n+1}$ with $n+1\ge 3$. In this
setting,  the CR singularity must have CR dimension $n$ and
CR points have CR dimension $n-1\ge 1$. For $M$ to bound a Levi-flat
submanifold $\widehat{M}$, a CR point must be CR non-minimal, namely, for each CR point $q\in M$, there is a proper CR
submanifold in $M$ passing through $q$ of  CR dimension $n-1$, which is in fact the transversal intersection of a leaf of the Levi-foliation in
$\widehat{M}$ with $M$. 
A solution for the local version of such a generalized  Bishop problem was
obtained in Fang-Huang \cite{FH} when $M$ is real analytic and has at least one elliptic direction, while the same problem remains open for $M$ being  just smooth.
While more explanation of terminologies  will be given  in \S 2, roughly speaking, a CR singular point $p$ of $M$  is called elliptic if after a certain biholomorphic change of coordinates, the projection of $M$  to the real hyperplane approximating $M$ to second order  is  a strongly convex hypersurface in that hyperplane. Equivalently, in an appropriate   holomorphic coordinate system $(z,w)\in \CC^n\times \CC$ which has $p$ corresponding to $0$ and has $M$ defined by $w=Q(z,\ov{z})+O(|z|^3)$ with $Q(z,\ov{z})$ a quadratic polynomial in $(z,\ov{z})$ we can achieve  that $Q(z,\ov z)\ge C|z|^2$ for a certain $C>0$.  A CR singular point $p\in M$ is said to have an elliptic direction if the intersection of $M$ with a certain transversal affine  complex two-space passing through $p$ has  an elliptic CR singularity at $p$. Namely, the ellipticity  is a condition on  the strong convexity in a certain subspace  (after a biholomorphic change of coordinates).


The global Bishop problem of bounding a Levi-flat hypersurface $\wh{M}$ by a compact real codimension two smooth submanifold $M$ with only non-minimal  CR points and two elliptic CR singular points was first studied in \cite{DTZ0, DTZ1} where the problem is treated as a parametrized Harvey-Lawson filling problem \cite{HL}. 
$\wh{M}$ was only proved to exist in the sense of distribution (a current). In a subsequent paper \cite{DTZ2}, Dolbeault-Tomassini-Zaitsev  showed that $\wh{M}$ is  smooth up to $M$ away from two elliptic points and is Lipschitz smooth up to CR singular points when
$M\subset\mathbb C^{n+1}=\CC^{n}\times \RR\times i\RR$ is the graph of 
a compact strongly convex hypersurface in $\CC^{n}\times \RR$.  While a certain convexity is needed even in the Harvey-Lawson filling problem to have smoothness  $\wh{M}$ near $M$, as in the two-dimensional setting \cite{BG,FM}, one expects the convexity should be weakened to be a holomorphic convexity instead of a geometric convexity and the graph position assumption for $M$ should be relaxed.  Indeed, it is this problem which motivates the current work and we give a complete solution along these lines. Our theorem below answers questions left open   from the work of \cite{DTZ0,DTZ1,DTZ2}  and can be regarded as a solution of a generalized Bishop problem for a $2n$-manifold in $\CC^{n+1}$ or more generally in a Stein manifold of complex dimension $(n+1)$: 

\begin{theorem} \label{maintheorem-global}
Let  ${\mathcal M}$ be a Stein manifold of complex dimension $n+1\ge
3$.  Assume  that 
\begin{enumerate}[label=({\bf{\Alph*}})]
\item  $M\subset {\mathcal M}$ is  a compact $C^{N_0}$-smooth submanifold of real dimension $2n$ with precisely two
elliptic CR singular points, where $N_0\ge 6$ is either a natural number or  $\infty$ or $\omega$;

\item $M$ is contained in the
boundary of a $C^2$-smoothly bounded strongly pseudoconvex domain; 

\item every CR point of $M$ is non-minimal.
\end{enumerate}
Then $M$ bounds a
unique $C^{\ell(N_0)}$-smooth Levi-flat real hypersurface with boundary  $\wh{M}\subset \mathcal M$, that has $M$ as its $C^{\ell(N_0)}$-smooth boundary, such that   the following holds:
\begin{enumerate}[label=({\bf{\Roman*}})]

\item $\wh{M}\sm M$ is foliated by a family of Stein submanifolds,   each of which is diffeomorphic to an open unit ball in $\CC^n$; 

\item the foliations in $\wh M$ shrink down to two elliptic points in the sense that  there is a $C^{\ell(N_0)}$-smooth function $t$ from $\wh{M}$ onto $[-1,1]$  that is constant on each leaf and has different value on different leaf such that $t^{-1}(\{-1, 1\})$ are precisely  the two elliptic CR singular points and thus  for any  neighborhood of $t^{-1}(-1)$ (or $t^{-1}(1)$) there is an $\e>0$ such that any leaf with its $t$-value satisfying  $|t+1|<\e$ (or $|t-1|<\e$, respectively) is contained in that neighborhood; 

\item $\widehat{M}$ is the  hull of holomorphy of $M$ when $N_0\ge 10$.
\end{enumerate}
Here for  $N_0=\infty$ or $\omega$, $M$ is  assumed to be smooth or real analytic, respectively; $\ell(N_0)=N_0$ when $N_0=\infty$ or $\omega$ and $\ell(N_0)=[\frac{N_0}{2}]-2$ when $N_0<\infty$.
\end{theorem}
\begin{remark}
The number of elliptic CR singular points cannot exceed two if all CR singular points of $M$ are elliptic.  Otherwise, as in \cite{DTZ1}, by removing small connected
saturated neighbourhoods of all  elliptic CR singular points, we obtain a compact manifold $\mathcal S$ with boundary with the given foliation of codimension one by its CR orbits. Applying the Reeb-Thurston Stability Theorem \cite{Th}, one conclude that $\mathcal S$ is diffeomorphic to $S^{2n-3}\times[0, 1]$, which is a contradiction.  
\end{remark}
The following example serves as   the   model case   in our main theorem:

\begin{example} Let $M=\{(z,w)\in \CC^{n}\times \CC:   u={\rm Re}(w),  v={\rm Im}(w), \ |z|^2+u^2=1,\ v=0\}$. Then $p=(0,-1)$ and $q=(0,1)$ are  two elliptic CR singular points of $M$ and all other points in $M$ are non-minimal CR points. $M$ is contained in
the unit ball in $\CC^{n+1}$. $M$ bounds a unique real analytic Levi-flat hypersurface $\wh{M}=\{(z,w)\in \CC^{n}\times \CC:   u={\rm Re}(w),  v={\rm Im}(w), \ |z|^2+|w|^2\le 1, v=0\}$ which has $M$ as its real analytic boundary. $\wh{M}\sm M$ is foliated by a family of complex submanifolds $\{\wh{M_t}\}_{-1<t< 1}$ where $\wh{M_t} =\{(z,t): z\in \CC^n, |z|< \sqrt{1-t^2}\}$ shrink down  to elliptic CR singular points $p$ and $q$.
\end{example}
Our main theorem can be regarded as a solution to a complex Plateau problem. Indeed,  when $\h{dim}_\RR M>1$ is odd, the Plateau problem would be asking if $M$  bounds a
complex submanifold or a complex analytic variety. A necessary condition for an affirmative solution is that $M$ has to be a CR manifold of hypersurface type as the boundary of a complex submanifold carries a naturally inherited CR structure. The solution was indeed the main content of a celebrated paper of Harvey-Lawson \cite{HL} and its later development by Dinh in \cite{Di}. 
The smoothness of $\wh{M}$ even near $M$  requires a certain global 
convexity of $M$ such as $M$ being contained in the boundary of a compact strongly pseudoconvex hypersurface.  When  $\h{dim}_\RR M$ is even,
$\wh{M}$ is of real odd dimension and what we can at most expect is for $M$ to bound a smooth submanifold foliated by maximally complex submanifolds.  Theorem \ref{maintheorem-global} serves a solution along these lines.
   
As mentioned before, the non-minimality at CR points is necessary for the existence of $\wh{M}$, just as the  CR condition is needed for an odd dimension manifold to be the regular  boundary of a complex manifold.  The existence of two elliptic points on $M$, 
whose definition will be explained later,  is also crucial for us to do the local construction of  $\wh{M}$ in the present work. Its existence on a manifold  is often derived  as a consequence of index theorem. The idea of using elliptic complex tangency to construct small leaves and then push them globally to solve the complex Plateau problem dates back to the original paper of Bishop \cite{Bis} and then was developed further in the work of Bedford-Gaveau \cite{BG}.

The problem of finding a Levi-flat hypersurface with a prescribed boundary is also formulated and studied extensively from the perspective of degenerate PDEs when $M$ is in the graph form. Here we mention 
a  paper by Slodkowski and  Tomassini \cite{ST}. (See also many references therein for  related works):

Let $M_0\subset  \CC^n\times \RR$ be a compact  real hypersurface  bounding a domain $D_0\subset \CC^n\times \RR$ such that $M_0\times i\RR$ is strongly pseudoconvex in $\CC^{n}\times\CC$. Let $v=\phi_0$ be a smooth real-valued function on $M_0$. 
Write $M$ for the graph of $\phi_0$, i.e., $M:=\{(z,z_{n+1}=u+iv):  (z,u)\in M_0,\ v=\phi_0\}$.  Let $\psi\in C^{2}(\ov{D_0})$ be a real-valued function in $(z,z_{n+1})$. 
The Levi-operator    applied  to  $\psi$ is defined as
\begin{equation*}
\mathcal L(\psi)=\h{det}\left(\begin{matrix}0&\psi_{z_j}\\
\psi_{\ov{z_k}}&\psi_{z_j\ov{z_k}}
\end{matrix}\right)_{1\le i,j\le n+1}.
\end{equation*}
There is a weak version of Levi convexity for any continuous function $\psi=-v+\phi(z,\ov{z},u)$. When  $\phi$ is  $C^2$-smooth, this weak convexity  is equivalent to the non-negativity of eigenvalues of the Levi form of the hypersurface defined by  $\psi=-v+\phi(z,\ov{z},u)=0$.  (See \cite{ST} for a detailed definition).
The  Levi Monge-Amp\'ere  equation is to solve the following boundary value problem:
\begin{equation}\label{1a}
\left\{\begin{aligned}
&\mathcal L(-v+\phi)=0\,\,\,\,\,\,\, {\rm in\,\,} D_0\\
&\ \ \ \ \ \ \ \ \ \phi=\phi_0\,\,\,\,\,\,\,{\rm on\,\,} M_0\\
&\ \ \ \ \ \ -v+\phi\,\,\,\ \ {\rm is\,\, Levi\,\, convex}.
\end{aligned}\right.  
\end{equation}
Now finding a Levi-flat hypersurface with boundary $M$ is equivalent to solving the Levi equation (\ref{1a})  in the classical sense. While the weak solution of (\ref{1a}) always exists and is unique \cite{ST}, among other things,  a least requirement for obtaining a $C^2$-smooth solution is that   CR points of $M$  are non-minimal as we mentioned before. Hence, a $C^2$-smooth solution   does not exist in general. Our theorem then provides a natural geometric setting where the unique solution $\phi$ is indeed smooth. Note that even in this setting,  our result is stronger than what is proved in \cite{DTZ2} as they need $D_0$ to be strongly convex.

Our proof is divided into proving four more or less independent results all of which may have applications in other contexts. In $\S 2$, We  first study the local problem, which is carried out through a careful analysis of singular foliations near an elliptic singularity. In this regard, we prove the following theorem:
\begin{theorem}\label{maintheorem-local}
Let $M\subset\CC^{n+1}$ be a $C^{N_0}$-smooth real codimension two submanifold with $n\geq 2$ and $p\in M$ an elliptic CR singular point. Assume that CR points in $M$ are non-minimal.  Here $N_0\ge 6$ is either a natural number or  $\infty$ or $\omega$. 
Then there is a $C^{\ell(N_0)}$-smooth Levi-flat hypersurface with boundary $\wh{M}$ that has a neighborhood of $p$ in $M$ as part of its boundary and that is foliated by compact  smooth complex hypersurfaces shrinking down to $p$. Moreover, there is a neighborhood of $p$ in $M$ such that each leaf (CR orbit) in  ${M}$ intersecting that neighborhood is CR $C^{\ell(N_0)}$-diffeomorphic to a compact strongly convex hypersurface in $\CC^n$. $\wh{M}$ is the local hull of holomorphy of $M$ near $p$.  Here $\ell(N_0)$ is  defined as in Theorem \ref{maintheorem-global}.
\end{theorem}
In the above, $\wh{M}$ is  the local hull of $M$ at $p$  in the sense that there is a decreasing sequence of  pseudoconvex domains $\{\O_\a\}$  containg  $p$ with $\cap_{\a} \O_\a =\{p\}$ such that the germ of the    hull of holomorphy of $\O_\a\cap M$ is the same as that  of $\wh{M}$ at $p$ for each $\a$.

\begin{example} \label{optimal-001}
When $N_0<\infty$ in Theorem \ref{maintheorem-local}, the $C^{\ell(N_0)}$-regularity  with $\ell(N_0)=[\frac{N_0}{2}]-2$ of $\wh{M}$ is  more or less optimal, as the following example  modified from the one  in \cite {Hu1} shows:

Let $l\geq2$ and $$M:=\{(z, w)\in \CC^{n}\times \CC : w = |z|^2 + |z|z_1^l\},\ z=(z_1,\cdots,z_n).$$ Then $M$ is  a real codimension two submanifold of $C^l$-smoothness near the elliptic CR singular  point $0$. The family of  closed complex balls centered at the origin with radius  $0<r\ll 1$ is attached to  $M$ near $0$  through the map defined by  $\varphi (\xi, r) = (r\xi, r^2 + r^{l+1}\xi^l_1),$ where $\xi=(\xi_1,\cdots,\xi_n)$ with $|\xi|\le 1$.
The local holomorphic hull $\wh{M}$  of M  near $0$ is the set $$\{(z = x+ iy,u+iv):z = r\xi,u = r^2 + r{\rm Re}(z_1^l),v = r{\rm Im}(z_1^l),\ \forall\,  |\xi|\le 1, 0\leq r\ll1\}.$$ 
The tangent space of $\wh{M}$ at $0$ is defined by $v=0$ and $\wh{M}$ is 
the graph of the function $v = v(x,y,u)$ over $\pi(\wh{M}) $ with $\pi$ the projection to the $(z,u)$-space. Then the regularity of $\wh{M}$ as a submanifold is the same as that for the graph function $v=v(x,y,u)$. Let $l \equiv  1$ mod $4$.  Along $x = 0$  and   $(y_2,\cdots,y_n)=0$,    $v(0, y_1,0, u) = \sqrt{u}y_1^l$ 
with $u \ge  y_1^2$. Then for $k >\frac{l+1}{2}$,
$\frac{\p^k v(0,y_1,0,u)}{\p u^k}|_{(y_1,u)=(y_1,y_1^2 )}$ is unbounded as $y_1\ra 0^+.$  Hence, we see that $\wh{M}$  is of $C^{\frac{l+1}{2}}$-smooth near $0$. 

Next we construct a small $C^l$-smooth strongly convex domain $D\subset \pi(\wh M)$ in the $(z,u)$-space with $M$ near $0$ as part of its boundary and has only one more point in $\p D$ that is tangent to a complex hyperplane defined by $u=\h{constant}, v=0$.  Then $\wh{M}\cap \pi^{-1}(\ov{D})$ is the Levi-flat hypersurface bounded by $\wh{M}\cap \pi^{-1}(\p {D})$ that is only $C^{\frac{l+1}{2}}$-smooth up to the boundary. This gives a global  example in which one can at most expect $C^{\frac{N_0+1}{2}}$-smoothness for $\wh{M}$ in Theorem \ref{maintheorem-global}.
\end{example}

In $\S 3$,  after we construct the local foliation near an elliptic  singular point, we push the foliation away from the CR singular point. The next step is then to prove an openness result. 
In this regard, we prove the following:

\begin{theorem} \label{openness} Let $\O$ be a $C^2$-smoothly bounded strongly pseudoconvex domain in a Stein manifold $\mathcal M$ of complex dimension $(n+1)$ with $n\geq 2$.  Assume that
$S\subset \p\O$  is a $C^{N_0}$-smooth  submanifold foliated by a $C^{N_0-2}$-smooth family of  mutually disjoint compact strongly pseudoconvex CR manifolds of CR dimension $(n-1)$
with a $C^{N_0-2}$-smooth diffeomorphism $\psi: (-1,1)\times S_0\ra S$  sending each  $\{t\}\times S_0$ to  a leaf $S_t\subset S$ .
Write  $\wh{S_t}$ for  the  complex analytic subvariety of codimension one   bounded by  ${S_t}$ for each $t$.  Assume that $\wh{S_0}$ is a Stein submanifold. Then each $\wh{S_t}$ is a Stein submanifold for $|t|<\e\ll1$ and  $\widehat S=\cup_{t\in (-\e,\e)} \widehat S_t$ is a $C^{\ell(N_0)}$-smooth real  hypersurface with a certain neighborhood of $S_0$ in $S$ as part of its $C^{\ell(N_0)}$-smooth boundary. Moreover the $t$-function on $\wh{S}$ which assigns value $t$ to the points in $\wh{S_t}$ is $C^{\ell(N_0)}$-smooth on $\wh{S}$ and $dt|_{\wh{S}}\not =0.$ Here $\ell(N_0)$ is defined as in Theorem \ref{maintheorem-global}.
 \end{theorem}

The next step is to prove a closeness result. While proofs in the first two steps are based on delicate analysis related to the complex analysis of several variables,  the proof for the closeness is motivated by a deep result of Eliashberg-Floer-McDuff \cite{McDuff} in symplectic geometry. Recall that a germ of complex analytic variety $(X,p)$ with $\dim_{\mathbb C}X=n$ is said to be an isolated hypersurface singularity if $X\sm\{p\}$ is smooth and $(X,p)$ is isomorphic to a germ of complex analytic subvariety of codimension one in $\mathbb C^{n+1}$. 
We prove in \S 4 the following:

\begin{theorem}\label{nosing} Let $W$ be a Stein space of dimension $n\geq 2$ with at most  finitely many isolated hypersurface singularities. Assume further that $W$ has a compact smooth strongly pseudoconvex boundary denoted by $\partial W$. Denote by $\xi$ the contact structure on $\partial W$ induced from its CR structure. If $(\partial W,\xi)$ is contactomorphic to the unit sphere  $(S^{2n-1}, \xi_{std}) $ in $\CC^n$ equipped with the standard contact structure, then $W$ is a smooth Stein manifold and is diffeomorphic to the complex unit ball $\BB^n\subset \CC^n$.
\end{theorem}

We prove Theorem \ref{nosing} and its generalization to isolated complete intersection singularities (Theorem \ref{icis}) along the line of studying the symplectic fillings of $(S^{2n-1},\xi_{std})$.  With  a complete understanding of  the symplectic fillings of $S^3$ from  the famous  works of Gromov \cite{Gromov}, Eliashberg \cite{Eli89} and McDuff \cite{McDuff90}, when $n=2$  the singularities in Theorem \ref{nosing} do not have to be of hypersurface type
and the boundary
$\partial W$ needs only to be diffeomorphic to $S^3$. More detailed discussions and connections to other related studies in the literature will be given in \S 4, where we make a conjecture generalizing \Cref{nosing}.

\medskip

The final step in the proof of Theorem \ref{maintheorem-global} is to show that the $\wh{M}$ constructed in previous Theorems is the hull of holomorphy of  $M$ in the sense that $\wh{M}$ is precisely the intersection of all pseudoconvex domains in $\mathcal M$ that contains $M$.

\begin{theorem}\label{hull} Let $M$ be as in Theorem \ref{maintheorem-global}. Assume $N_0\ge 10$. Then $\widehat M$ is the hull of holomorphy of $M$.
\end{theorem}

\section{Local theory:  Analysis of singular foliations  near an elliptic CR singularity}

\subsection{Formal flattening of an elliptic CR singularity}

Assume that $M\subset \CC^{n+1}$ is a $C^{N_0}$-smooth submanifold of real codimension two and $p\in M$ is a CR singular point. Write $(z,w)=(z_1,\cdots,z_n,w)$ for the coordinates of
${\mathbb C}^{n+1}$ with $N_0\ge 6$ being a positive integer or $N_0=\infty$ or $N_0=\o$ (namely, $M$ is real analytic). After a holomorphic change of coordinates, we
assume that $p=0,\ T^{(1,0)}_pM=\{w=0\}$. Then $M$ near $p=0$ is the
graph of a function of the form:
\begin{equation}\label{001}
 w=Q(z,\ov{z})+O(|z|^3),
 \end{equation}
where $Q(z,\ov{z})$ is a homogeneous polynomial of degree two in
$(z,\ov{z})$ and $|z|:=\sqrt{\sum_{j=1}^n|z_j|^2}$. 

After a simple holomorphic change of
coordinates, if needed, we can  write
\begin{equation} \label {0011}
Q(z,\ov{z})=2{\rm Re}{(z\cdot \mathcal A\cdot z^t)}+z\cdot \mathcal B\cdot \ov{z}^t
\end{equation}
with $\mathcal A$ and $ \mathcal B$ being two $n\times n$ matrices. Suppose that $$z=\wt{z}\cdot
P+\wt{w}\overrightarrow{a}+
O\left(|w|^2+|z|^2\right),\,\,\ w=\mu \wt{w}+ \overrightarrow b\cdot z^t+
O\left(|w|^2+|z|^2\right)$$  is  a holomorphic transformation preserving the form
as in (\ref{001}) and (\ref{0011}). Then $\overrightarrow b=0$, $\mu\not =0$, and $P$
is an $n\times n$ invertible matrix. Moreover, if $(M,0)$ is
defined in the new coordinates by $\wt{w}=\wt{Q}(\wt{z},\ov{\wt{z}})+O(|\wt{z}|^3)$ with $\wt{Q}(\wt{z},\ov{\wt{z}})=2{\rm Re}({\wt{z}\cdot \wt{\mathcal A}\cdot
\wt{z}^t})+\wt{z}\cdot \wt{\mathcal B}\cdot \ov{\wt{z}}^t,$ then
$\wt{\mathcal B}=\mu^{-1} P\cdot \mathcal B\cdot \ov{P}^t$, $\wt{\mathcal A}=\ov\mu^{-1} P\cdot\mathcal A\cdot {P}^t$. Hence the non-degeneracy of   $\mathcal B$ is an invariant notion. We say $p=0$ is a non-degenerate  CR singular point if $\mathcal B$ is invertible. For the  $M$ that interests us, any  CR point $q\in M$ is assumed to be non-minimal, namely,
{\it for each CR point $q\in M$, there is a proper CR
submanifold  in ${\mathbb C}^{n+1}$ passing through $q$, that is
contained in $M$ and has  CR dimension $n-1$}. By a result of Fang-Huang in \cite{FH},  there is a holomorphic change of coordinates near $p=0$ such that $\mathcal B$ is Hermitian.  When $\mathcal B$ is further assumed to be positive definite, then there is a holomorphic change of 
coordinates in which
\begin{equation}\label{qm}
Q(z,\-{z})=\sum\nolimits_{j=1}^{n}\left(|z_j|^2+\lambda_j(z_j^2+\-{z_j}^2)\right),\,\, 0\le \ld_1\le \ld_2\le\cdots\le \ld_n<\infty.    
\end{equation}
The set  $\{\ld_1,\cdots,\ld_n\}$ is a holomorphically invariant set and is called the set of Bishop invariants of $M$ at $p$.  

\begin{definition}
$p$ is called an elliptic CR singular point of $M$ if each $\ld_j\in [0,1/2)$. Namely,  there is a holomorphic change of coordinates in which  $p=0$ and the first non-vanishing term of a defining function 
$(M,p=0)$ is a positive definite quadratic polynomial. 
\end{definition} 
Ellipticity is equivalent to its quadratic model being strongly convex near the CR singular point in certain coordinates.

In what follows, we always assume  CR points of $M$ are non-minimal. 
Then by  a result of Huang-Yin in \cite{HY3},  when $p=0$ is an elliptic CR singular point and $M$ is smooth,  $M$ can be  flattened near $p=0$ to any order $2<N\le N_0$ with $N<\infty$
in the sense that, there is a holomorphic change of  coordinates such that in the new coordinates, $M$ is defined in a neighborhood of $p=0$ by an equation of the form
\begin{equation}\label{010-01}
w=w(z,\ov{z})=G(z,\ov{z})+iE(z,\ov{z}), \ G=\sum_{j=1}^{n}\left(|z_j|^2+\lambda_j(z_j^2+\-{z_j}^2)\right)+O(|z|^3),\ E=o(|z|^N),
\end{equation}
where  both $G(z,\ov{z})$ and $E(z,\ov{z})$ are real-valued $C^{N_0}$-smooth functions. 
Moreover, when $M$ is real analytic there is a holomorphic change of coordinates \cite{HY3} such that in the new coordinates $E\equiv0$.
Starting from this flattening result, we will prove in this section Theorem \ref{maintheorem-local}.

\subsection{Integrability condition}

In the remaining of the section, we fix an $N\in \NN$ with $6\le N\le N_0$ and  let $(M,p=0)$ be a $C^{N_0}$-smooth submanifold of real codimension two in
${\mathbb{C} }^{n+1}$ as defined in (\ref{010-01}).  Here, when $N_0<\infty$, we take $N=N_0.$
Still write 
\begin{equation} \label{7-label}
w=w(z,\ov{z})=G(z,\ov{z})+iE(z,\ov{z})=Q(z,\ov{z})+P(z,\ov{z})+iE(z,\ov{z}), 
\end{equation}
where
$Q(z,\ov{z})$ is given by (\ref{qm}) with $0\le \lambda_1,\cdots, \lambda_n<\frac{1}{2}$, and $P(z,\ov{z})$, $E(z,\ov{z})$ are real-valued $C^{N_0}$-smooth
functions such that $P=O(|z|^{3})$, $E=o(|z|^N)$.

For $1\leq j,k\leq n$, we define vector fields on $M$
\begin{equation}\begin{split}\label{325eq2}
&L_{j,k}:=A_{(k)}\frac{\p}{\p z_j}-B_{(j)}\frac{\p}{\p z_k}+C_{(j,k)}\frac{\p}{\p
w}\\
&:=\left(\frac{\p G}{\p z_k}-\sqrt{-1}\frac{\p E}{\p z_k}\right)\frac{\p}{\p z_j}-\left(\frac{\p G}{\p z_j}-\sqrt{-1}\frac{\p E}{\p z_j}\right)\frac{\p}{\p
z_k}+2\sqrt{-1}\left(\frac{\p G}{\p z_k}\frac{\p E}{\p z_j}-\frac{\p G}{\p z_j}\frac{\p E}{\p z_k}\right)\frac{\p}{\p w}.
\end{split}\end{equation}
Computation yields that 
\begin{equation}\label{325eq3}
A_{(k)}=\ov{z_k}+2\lambda_k{z_k}+O(|z|^2),\ B_{(j)}=\ov{z_j}+2\lambda_j{z_j}+O(|z|^2),\
C_{(j,k)}=o(|z|^{N}).
\end{equation}
In particular, $\{L_{j,k}\}_{j,k=1}^{n}$ generates  complex tangent vector fields of type
$(1,0)$ along $M$ near $0$ but not including $0$. 
By  computation carried out  in  \cite{HY4}, we further  have for any $1\leq j\neq k\leq n$,
\begin{equation}\label{326eq2}
[L_{j,k},\ov{L_{j,k}}]
=\lambda_{(1jk)}\frac{\p}{\p
\ov{z_j}}+\lambda_{(2jk)}\frac{\p}{\p
\ov{z_k}}+\lambda_{(3jk)}\frac{\p}{\p
\ov{w}}+\lambda_{(4jk)}\frac{\p}{\p z_j}+\lambda_{(5jk)}\frac{\p}{\p
z_k}+\lambda_{(6jk)}\frac{\p}{\p w},
\end{equation}
where
\begin{equation}\label{326eq5}
\begin{array}{rclrcl}
     \lambda_{(1jk)} & = &-\ov{z_j}-2\lambda_j{z_j}+O(|z|^2), & \lambda_{(2jk)} & = &-\ov{z_k}-2\lambda_k{z_k}+O(|z|^2),\\
     \lambda_{(4jk)} &= & {z_j}+2\lambda_j\ov{z_j}+O(|z|^2), & \lambda_{(5jk)} & = & z_k+2\lambda_k\ov{z_k}+O(|z|^2), \\
     \lambda_{(3jk)} &= & o(|z|^{N}), & \lambda_{(6jk)} & = & o(|z|^{N}).
\end{array}
\end{equation}
Since we have assumed  that $M$ is non-minimal at its
CR points, $$\big\{\h{Re}(L_{j,k}), \h{Im}(L_{j,k}), \h{Im}([L_{j,k},\ov{L_{j,k}}])\big\}_{j,k=1}^{n}$$  is an integrable distribution on $M$ away from the CR singular point, whose leaves are precisely  CR submanifolds of hypersurface type in the CR part of  $M$ of CR dimension $(n-1)$. We will call each connected leaf a CR orbit of the distribution.

\subsection{ CR orbits near an elliptic CR singular point}
We now present a proof of Theorem \ref{maintheorem-local}. When $M$ is real analytic, the proof follows readily from the result of Huang-Yin in \cite{HY3} (or \cite{HY5} whose results were later published in  \cite{HY3} and \cite{HY4}). Notice that the Lipschitz continuity up to the singular point was obtained in  \cite{DTZ2}. Based on the argument developed in the analytic category \cite{HY5}, the $C^{\infty}$-smooth case  was also considered  in \cite{Bur2}. The argument in \cite{HY5} works well in the real analytic category but is hard to be adapted to the smooth category. The proof presented here 
is based on a careful analysis of the behavior near an elliptic CR singularity for  CR foliations, which works equally well in the finite smoothness case with an effective estimate on the regularity of the hull.  More importantly  our construction and the   regularity  of $\wh{M^a_{r^*}}$ in (\ref{equ-0100}) are  heavily needed later to construct a good defining function of $\wh{M}$ near singular points that is crucial to construct a global Stein neighborhood basis for $\wh{M}$.   The analysis in this section is partially motivated by the important work of Dolbeault-Tomassini-Zaitsev in \cite{DTZ2} without appealing to the holomorphic disk method that was crucial  in the work of Kenig-Webster \cite{KW1}, Huang-Krantz \cite{HK} and Huang-Yin \cite{HY3}.  In particular,  our proof of the closeness of the orbits near an elliptic CR singular point is partially  motivated by a beautiful  idea of Dolbeault-Tomassini-Zaitsev in \cite{DTZ2}.

A basic idea in our compliacted analysis is to cover the punctured manifold $M\sm\{0\}$ with infinitely many open balls where uniform estimates can be achieved after a rescaling.  Such a cover resembles to a cover with finite geometry in other geometric contexts. 

Assume that $p=0$ and $M$ is as defined in (\ref{7-label}).   Consider the projection map $\pi: M\ra \CC^n$ which sends $(z,w)\approx 0$ to $z$.  Then $\pi$ extends to a holomorphic map. Moreover, $\pi$, when restricted to a CR leaf (also called, in what follows, a CR orbit),  is a  CR diffeomorphism,  as its inverse is precisely the graph map $\Gamma:z\ra (z,w(z,\ov{z}))$. Now 
$${\mathcal E}:=\big\{X_{2j-1,k}:=\pi_{*}(\h{Re}(L_{j,k})), X_{2j,k}:=\pi_{*}(\h{Im}(L_{j,k})),X_{2n+1,k}:=\pi_{*}(\h{Im}([L_{k-1,k},\ov{L_{k-1,k}}]))\big\}_{j,k=1}^{n}$$  
is an integrable distribution in $\CC^n\sm \{0\}$ near $0$ whose leaves are images of   $C^{N_0-2}$-smooth CR orbits in $M\sm\{0\}$ under $\pi$. (Here we make the convention that $L_{0,1}=L_{n,1}$
when $k=1$, and that $N_0-2=\infty$, $\omega$ when $N_0=\infty$, $\omega$, respectively.) Recall that when $E\equiv 0$ in (\ref{010-01}), $M$ locally bounds the Levi-flat hypersurface $\{(z,w)\in\mathbb C^{n+1}:{\rm Re}\, w\geq Q(z,\ov{z})+P(z,\ov{z}),\,\,{\rm Im}\, w=0\}$. We write the corresponding distribution and  vector fields as
$$\E^0:=\big\{ X_{2j-1,k}^0, X_{2j,k}^0, X_{2n+1,k}^0\}_{j,k=1}^{n}.$$
Note that $\E^0$
is an integrable distribution  of $C^{N_0}$-smootness.
Write  $|X(z)|:=\sum_{j=1}^{n} (|a_j(z)|+|b_j(z)|)$ for any vector field $X=\sum_{j=1}^{n}(a_j\frac{\p} {\p x_j}+b_j\frac{\p} {\p y_j})$ on $\CC^n$, where $z_j=x_j+\sqrt{-1}y_j$. 

Let $\{V_k\}_{k=1}^{n}$ be an open cover of $\CC^n\sm \{0\}$, where each $V_k$ is an open cone defined by
\begin{equation}\label{cone}
V_{k}:=\left\{z\in\CC^n: \ |z_k|^2>\frac{1}{n+1}|z|^2=\frac{1}{n+1} \sum\nolimits_{j=1}^n(x_j^2+y_j^2)\, \right\}.    
\end{equation} 
Notice that by (\ref{325eq2})-(\ref{326eq5}),  there is a ball $B(0,r^*)\subset\mathbb C^n$, which is centered at $0$ with radius $0<r^*\ll1$, and a positive constant $C_0$ such that the following holds. For each $1\leq k\leq n$, $1\leq j\leq 2n+1$ with $j\neq 2k-1,2k$, and $z\in V_{k}\cap B(0,r^{*})$,
\begin{equation}\label{smallp}
{C_0}^{-1}|z|\le |X^0_{j,k}(z)|\le C_{0}|z|,\ \  |X_{j,k}(z)-X_{j,k}^0(z)|= o(|z|^{N-1});    
\end{equation}
for each $1\leq k\leq n$, the vector fields
$\{X_{j,k}\}_{\substack{1\leq j\leq 2n+1\\{\rm with}\,\,j\neq2k-1,2k}}$, $\{X^0_{j,k}\}_{\substack{1\leq j\leq 2n+1\\{\rm with}\,\,j\neq2k-1,2k}}$ form bases of the distributions $\mathcal E$, $\mathcal E^0$, respectively, on $V_{k}\cap B(0,r^{*})$. Moreover, a derivative of order $\a$ with $|\a|\le N-2$ of a coefficient of $X_{j,k}(z)-X_{j,k}^0(z)$ is of vanishing  order $o(|z|^{N-1-|\a|})$ as $z\ra 0$.

Write $S^{00}=\{z\in\CC^n: Q(z,\ov{z})=1\}$, and $S^0_r=\{z\in\CC^n: G(z,\ov{z})=Q(z,\ov{z})+P(z,\ov{z})=r^2\}$ for $0<r\ll1$, where $Q,P$ are as in (\ref{7-label}). Note that $S^{00}$ and $S_r
^0$ with $0<r\ll1$ are boundaries of strictly convex domains containing $0$. 
In what follows, we shall fix a point
\begin{equation}\label{p0ra}
p_0=(0,\cdots,0,(1+2\ld_n)^{-\frac{1}{2}})\in S^{00}.    
\end{equation}
Since $r^2+P(rp_0,\ov{rp_0})$ is strictly increasing in $r$ for $0<r\ll1$, by the implicit function theorem, we derive a $C^{N_0}$-smooth function $r_a(r)=r+O(r^2)$  for $0< r<r^{*} \ll1$ such that 
\begin{equation}\label{p0ra2}
r_ap_0\in S^0_{r}.    
\end{equation}

We find a constant $0<\d\ll1$ and sufficiently many (more or less equally distributed)  points $\{p_l\}_{l=0}^{m-1}\subset S^{00}$ such that,  each  $B(r p_l,6r\d)$ is compactly supported in a certain  $V_{k}$, and $S^0_r\subset\cup_{l=0}^{m-1}B\left(rp_l, r\d\right)$.  Then  for $0<r^*\ll1$,\ $\{B\left(rp_l, r\d\right)\}_{0<r<r^*, 0\leq l\leq m-1}$ covers the domain $\O^0_{r^*}\sm\{0\}$, where $\O^0_{r^*}\subset\mathbb C^n$ is the domain enclosed by $S^0_{r^*}$.

Fix an $l$ with $0\leq l\leq m-1$ and parametrize the open subset ${B(p_l, 6\d)}\cap S^{00}$ of the ellipsoid $S^{00}$ by certain spherical coordinates $(\theta_1,\cdots,\theta_{2n-1})$. Denote by $\Theta_l:{B(p_l, 6\d)}\cap S^{00}\rightarrow U_l\subset\{\theta:=(\theta_1,\cdots,\theta_{2n-1})\in\mathbb R^{2n-1}\}$ the corresponding coordinates map.

Write $\d_r: z\in\CC^n\ra \CC^n$ for the dilation map defined by $\d_r(z)=rz$. We define a $C^{N_0}$-smooth coordinates map $H^{0\#}_{l,r}$ from $B(p_l, 5\d)$ into $U_l\times (-1,1)\subset\subset \RR^{2n}$
such that 
each point $q\in B(p_l, 5\d)$ is mapped to a $(\theta, x_{2n})$, where $\theta=\Theta_l(q^{00})$ with $q^{00}\in S^{00}$ and on the same radial ray of $q$, and  $x_{2n}$ is defined such that $ rq\in S^0_{(x_{2n}+1)r}$. In particular, the $x_{2n}$-component of $(H^{0\#}_{l,r}\circ\d_{r^{-1}})(S^0_r)$ is $0$. 
Then each radial ray (originated from the origin) is mapped to a vertical  line in $\RR^{2n}$ defined by $\theta=\h{constant}$ and 
each  $\d_{{r}^{-1}}(S_{(1+t){r}}^0)\cap B(p_l, 5\d)$ is mapped to a horizontal hyperplane defined by $x_{2n}=\h{constant}$. 
More precisely, the $x_{2n}$-component of the map $H^{0\#}_{l,r}(z)$ is given by $\sqrt{Q(z,\ov z)+r^{-2}P(rz,r\ov z)}-1$. Computation yields that for $|\a|\leq N$,
\begin{equation}\label{bg}
D_z^{\alpha}H^{0\#}_{l,r}=O(1),\ \ D_{(\theta,x_{2n})}^{\alpha}(H^{0\#}_{l,r})^{-1}=O(1),\,\,\h{as}\ r\rightarrow0^+.
\end{equation}
Here and in what follows we always use the notation that for any $2n$-tuple of nonnegative integers $\alpha=(\a_1,\a_2,\cdots,\a_n,\a_{\bar 1},\a_{\bar 2},\cdots,\a_{\bar n})$,
\begin{equation*}
D_z^{\alpha}:=\frac{\partial^{\a_1}}{\partial z_1^{\a_1}}\frac{\partial^{\a_2}}{\partial z_2^{\a_2}}\cdots\frac{\partial^{\a_n}}{\partial z_n^{\a_n}}\frac{\partial^{\a_1}}{\partial \ov{z_1}^{\a_{\bar 1}}}\frac{\partial^{\a_{\bar 2}}}{\partial \ov{z_2}^{\a_{\bar 2}}}\cdots\frac{\partial^{\a_{\bar n}}}{\partial \ov{z_n}^{\a_{\bar n}}}.  
\end{equation*}

Fix an integer $k$ such that $B(r p_l,6r\d)\subset\joinrel\subset V_{k}$. Now for any $j\neq2k-1, 2k$,
\begin{equation*}
(H^{0\#}_{l,r}\circ\d_{r^{-1}})_{*}(X_{j,k}^0)=\sum\nolimits_{\alpha=1}^{2n-1}b_{ljk,r}^{\alpha}(\theta,x_{2n})\frac{\p}{\p \theta_{\alpha}}\,\,\,\,\h{for}\,\, (\theta,x_{2n})\in U_l\times(-1,1).
\end{equation*} 
By considering the leading terms of $X_{j,k}^0$ given by the quadratic model $w=Q(z,\ov z)$, a direct computation yields that 
\begin{equation*}
C^{-1}\cdot I_{(2n-1)\times (2n-1)}\leq  (b_{ljk,r}^{\alpha}(\theta,x_{2n}))_{\substack{1\leq j\leq 2n+1\,{\rm with\,}j\neq 2k-1,2k\\{\rm and\,\,}1\leq\alpha\leq 2n-1}}\leq C\cdot I_{(2n-1)\times (2n-1)}
\end{equation*}
over $U_l\times(-1,1)$ for a positive constant $C$ which is independent of $r$. By (\ref{smallp}),
the distribution ${\mathcal E}^{\sharp}_{l,r}:=(H^{0\#}_{l,r}\circ\d_{r^{-1}})_{*}({\mathcal E})$ has a basis of the following form on $U_l\times(-1,1)$
\begin{equation*}
X_{l\alpha,r}^{\sharp}=\frac{\p}{\p \theta_{\alpha}}+h_{l\alpha,r}(\theta,x_{2n})\frac{\p}{\p x_{2n}}\,\,\,\,\h{for}\,\,  1\leq\alpha\leq2n-1,   
\end{equation*}
where for any $|\b|\leq N-2$,  \begin{equation}\label{hn2}
D^\b_{(\theta,x_{2n})} h_{l\alpha,r}(\theta,x_{2n})=o(r^{N-2}).
\end{equation} 

Fix $r>0$ and a point $p^*_l$ with $rp_l^*\in B(rp_l, 5\d r)\cap S^0_r$. Denote by $\theta^*$ the spherical coordinates of $p^*_l$. Choose a $C^{N-2}$-smooth function $\xi(t)$ on $(-1,1)$ such that 
$D^{\a}_t(\xi(t)-t)=o(r^{N-2})$ for $|\a|\leq N-2$. Then, to find a leave of  ${\mathcal E}^{\sharp}_{l,r}$  given by the graph of $x_{2n}=x_{2n}(\theta)$ with $x_{2n}(\theta^*)=\xi(t)$, it suffices to solve the following system of partial differential equations:
\begin{equation*}
\left\{\begin{aligned}
&\frac{\p x_{2n}(\theta)}{\p \theta_{\alpha}}=h_{l\alpha,r}(\theta, x_{2n}(\theta))\,\,\,\h{for}\,\,  1\leq\alpha\leq 2n-1\\
&\,x_{2n}(\theta^*)\,\,=\,\xi(t)
\end{aligned}\right.
\end{equation*}
under the compatbility condition $\frac{\p h_{l\alpha,r}}{\p \theta_{\beta}}=\frac{\p h_{l\beta,r}}{\p \theta_{\alpha}}$ due to the integrability of  ${\mathcal E}^{\sharp}_{l,r}$. Note that to parametrize leaves one can simply set $\xi(t)=t$, and more general $\xi(t)$ will be used to make transitions across different foliated charts. 

By  the implicit function theorem, we  have a unique solution 
\begin{equation}\label{x-2n}
x_{2n}=\Xi_{l,r}(\theta, t),\ \Xi_{l,r}(\theta^*,t)=\xi(t),
\end{equation}
where $\Xi_{l,r}$ is  $C^{N-2}$-smooth in its variables. Moreover by (\ref{hn2}) we have
\begin{equation}\label{bgg1}
D^{\a}_{(\theta,t)}\left(\Xi_{l,r}(\theta, t)-\xi(t)\right)=o(r^{N-2})\ \ \h{for}\,\,|\a|\leq N-2.    
\end{equation} 
Define a graph map 
\begin{equation}\label{ggl}
\Gamma_{l,r}(\theta,t)=\left(\theta,  \Xi_{l,r}(\theta, t)\right).    
\end{equation}
Then
\begin{equation}\label{deltar}
\d_r\circ (H^{0\#}_{l,r})^{-1}\circ \Gamma_{l,r}\circ H^{0\#}_{l,r}\circ \d_{r^{-1}}    
\end{equation} maps $S^0_{(1+t)r}$ in $B(r p_j,3r\d)$ into an open piece of a leave of $\mathcal E$ in $B(r p_j,4r\d)$, whose intersection with $\gamma^*(t)$ has the last coordinate $\xi(t)$ after applying $H^{0\#}_{l,r}\circ \d_{r^{-1}}$. Here $\gamma^*(t)$ is a parametrization of part of the ray initiated from the origin and through $rp_l^*$ such that $\gamma^*(t)\in S^0_{(1+t)r}$.  

Solving for $t$ in (\ref{x-2n}) with $\theta$ a parameter, we have $t=\Xi_{l,r}^{-1}(\theta,x_{2n})=x_{2n}+o(r^{N-2})$. (Namely, $t=\Xi_{l,r}^{-1}(\theta,\Xi_{l,r}(\theta,t))$).
Write $H_{l,r}^0:=H^{0\sharp}_{l,r}\circ \d_{r^{-1}}.$ Then when restricted to  $B(rp_l, 3r\d)$, the map 
$$H_{l,r}(z):=(\theta, \Xi_{l,r}^{-1}(\theta,x_{2n}))\circ H_{l,r}^0(z)$$ gives a trivialization for $\mathcal E$. 
Computation by (\ref{bg}), 
(\ref{bgg1}) yields that for $|\a|\le N-2$, 
\begin{equation}\label{HH}
D_z^{\a}\left(H_{l,r}-H_{l,r}^0\right)=o(r^{N-2-|\a|}),\,\,D^{\a}_{(\theta,t)}\left(H_{l,r}^{-1}-(H^0_{l,r})^{-1}\right)=o(r^{N-1}) . 
\end{equation}

Next, we have the following:
 
\begin{lemma}\label{lemma2.3}
Let $0<r_0<r^*\ll1$ and $N\ge 4$. Let $\L,\ \L^0$ be leaves of $\E,\ \E^0$, respectively, in   $B(r_0p_l, 2 r_0 \d)$ for a certain $l$, such that the distance of a point on $\L^0$ and in the ray initiated from the origin to a point in $\L$ and in the same ray is of the order $o(r_0^{N-1})$. Then  $\L$ is strongly convex. Moreover the ray distance from  $\L$ to  $\L^0$ is of order $o(r^{N-1}_0)$. 
 \end{lemma}

Here 
the ray distance of $\L$ and $\L^0$ is defined as the maximum distance between two points on $\L$ and $\L^0$ in the same ray initiated from the origin, once we know both $\L$ and $\L^0$ are strongly convex with respect to the origin.
 \begin{proof}
Take a point $p^*_l$ with $r_0p_l^*\in B(r_0p_l, 3r_0\d)\cap S^0_{r_0}$. Set $\xi(t)=t$. Then, as discussed above, there are   diffeomorphisms  $H_{l,r_0}$
and $H^0_{l,r_0}$ from $B(r_0p_l, 3 r_0\d)$ into   $U_l\times (-1,1)$ with coordinates $(\theta,x_{2n})$ such that (\ref{HH}) holds and that $H_{l,r_0}$ (resp. $H^0_{l,r_0}$) maps each leaf of $\E$ (resp. $\E^0$) in $B(r_0p_l, 3 r_0\d)$ to a horizontal hyperplane 
in $U_l\times (-1,1)$. 


Suppose that $H_{l,r_0}(\L)=\{x_{2n}=c_1\}$ and $H^{0}_{l,r_0}(\L^0)=\{x_{2n}=c_0\}$. Then the given condition in the lemma yields that $c_1-c_0=o(r_0^{N-2})$.
Consider $\d_{r_0^{-1}}(\Lambda)$ and 
$\d_{r_0^{-1}}(\Lambda^0)$, 
of which the defining functions are given,  respectively, by $\rho(z)=x_{2n}\circ H_{l,r_0}\circ\d_{r_0}(z)-c_1$ and
$\rho_0(z)=x_{2n}\circ H^{0}_{l,r_0}\circ\d_{r_0} (z)-c_0$ (indeed $\rho_0(z)=r_0^{-1}\sqrt{Q(r_0z,r_0\ov z)+P(r_0z,r_0\ov z)}-c_0$). By (\ref{HH}), $D^{\a}_z(\rho(z)-\rho_0(z))\ra 0$ as $r_0\ra 0$ for $|\a|\le 2$. 
Notice that $\rho_0(\rho_0+2c_0)=Q(z,\ov z)+r_0^{-2}P(r_0z,r_0\ov z)-c_0^2$ is a strongly convex defining function of $\d_{r_0^{-1}}(\Lambda^0)$. 
Then $D^{\a}_z(\rho(\rho_0+2c_0)-\rho_0(\rho_0+2c_0))=o(1)$ as $r_0\ra 0$ for $|\a|\le 2$.
Thus  $\d_{r_0^{-1}}(\Lambda)$ has a strongly convex defining function $\rho(\rho_0+2c_0)$. Hence $\L$ is strongly convex.

Once we know $\L$ is strongly convex, the ray distance from $\L$ to $\L^0$ is well-defined and is apparently of quantity $o(r_0^{N-1})$ as $(H^0_{l,r_0})^{-1}(\theta, c_0)-H_{l,r_0}^{-1}(\theta, c_1)=o(r^{N-1}_0)$ by (\ref{HH}). \end{proof}

For any $r$ such that $0<r<r^*\ll1$, let $\L$, $\L^0$ be the leaves of $\mathcal E$, $\mathcal E^0$ passing through $r_ap_0\in S^0_{r}\cap B(rp_0, 2r\d)$, respectively, where $p_0$ and $r_a$ are given in (\ref{p0ra}) and (\ref{p0ra2}). 

We first find a constant $C$ such that any two points in $S_r^0$ can be connected by a piecewise smooth curve of length no larger than $Cr$. Then any two homotopic piecewise smooth curves relative to their endpoints can be connected by a continuous family of piecewise smooth curves of length no larger than $C^*r$ for a certain  $C^*>C$. 

Motivated by the work in \cite{DTZ2}, we now define a map $F_0$ from $\L^0$ into $\L$ as follows: For any curve $\gamma^0$ connecting $r_ap_0$ to $p\in \L ^0$ with length bounded from above by $Cr$,
we divide $\gamma^0=\gamma_1+\cdots+\gamma_\ell$ with $\ell$ being uniformly bounded such that $\gamma_j([t_j,t_{j+1}])\subset B(rp_{k(j)}, r\d)$ for a certain $k(j)$, $1\leq j\leq l$. Here $0=t_1<t_2<\cdots <t_{\ell+1}=1$. Also, the length of each $\gamma_j$ is bounded from above by $C_1r$.  Define  $F_0$ along $\gamma_1$ by sending $\gamma_1(t)$ to its ray  projection $P(\gamma_1(t))$ in $\L$. (Here all rays are initiated from the origin.)  Now there is a unique $\L_{2}$  such that $P(\gamma_1(t_2))\in \L_{2}$, where $\L_2$ is a certain leaf of $\mathcal E$ in    $B(rp_{k(2)}, 3r\d)$ for a certain $k(2)$. Also the intersection of $\L_2$ with the ray through $\gamma_1(t_2)$ has distance of order $o(r^{N-1})$ to the intersection of $\L^0$ with the ray through $\gamma_1(t_2)$  by Lemma \ref{lemma2.3}.  Next, we extend $F_0$ to  $\gamma_2$ by mapping $\gamma_2$ to its ray projection to $\L_2$.  Continuing in this manner, we define $F_0$ along $\gamma^0$.  Now a similar argument can be used to show that if $\gamma^1$ is also a piecewise smooth curve connecting $r_ap_0$ to $p$ in $S_r^0$ with length bounded  by $Cr$ and if $\{\gamma^t\}$ is a continuous family of piecewise smooth curves connecting $\gamma ^0$ to $\gamma^1$, relative to their endpoints, of lengths bounded by $C^*r$, then $F_0(\gamma^0(1))=F_0(\gamma^1(1))$. 
Since each $S_{r}^0$ is simply connected, we have a well-defined map $F_0$ from $S_{r}^0$ into a leave of $\mathcal E$  through $r_ap_0$.  Apparently $F_0$ is a local diffeomorphism with $C^{N-2}$-smoothness. Since  $S_r^0$ is compact, $F_0$ is an open and closed local diffeomorphism. 

In what follows, we always denote by $S_r$ the image of $S_r^0$ under  $F_0$. Then $S_{r}$ has to be compact and hence $F_0$ is a covering map. Moreover $S_r$ is strongly convex as it is locally. Thus $F_0:S^0_{r}\rightarrow S_{r}$ is a diffeomorphism  for $0<r\ll1$. 


Since each $S_r$ is a compact strongly convex hypersurface that is $o(r^{N-1})$-close to $S_r^0$,
we write $\Omega^0_{r}$ and $\Omega_{r}$ for  domains in $\CC^n$ enclosed by $S_{r}^0$ and $S_{r}$, respectively.   Write  
\begin{equation}\label{equ-0100}
\begin{split}
&\wh{M_{r^*}^0}:=\{ (z,u)\in \CC^n\times \RR: \ z\in \ov{\Omega_r^0},\  0< u=r^2<{r^*}^2\}\cup\{0\},\\
&\wh{M^a_{r^*}}:=\{(z,u)\in \CC^n\times \RR:  z\in \ov{\Omega_ {r}},\ 0< u=r^2<{r^*}^2\}\cup\{0\},
\end{split}    
\end{equation}
where $r^*$ is again a sufficiently small positive number.  

Let $F_0$ be constructed as above. By choosing suitable initial value functions $\xi_j(t)$ inductively along $\gamma_j$, it is easy to verify by (\ref{bgg1}) that $F_0$ is a $C^{N-2}$-smooth diffeomorphism from a neighborhood of $0\in\mathbb C^n$ but not including $0$ to another punctured neighborhood of $0\in\mathbb C^n$. Then $(F_0(z),u)$ defines a $C^{N-2}$-smooth diffeomorphism from $\p \wh{M^0_{r^*}}\sm \{0\}:=\{z\in S_r^0,\  0< u=r^2<{r^*}^2\}$ to $\p \wh{M^a_{r^*}}\sm \{0\}:=\{z\in S_ {r},\ 0< u=r^2<{r^*}^2\}$.
We next prove the following:  \begin{proposition}\label{mamodel} 
$(F_0(z),u)$ extends to a $C^{N-2}$-smooth diffeomorphism $F$ from $\wh{M_{r^*}^0}\sm\{0\}$ to $\wh{M_{r^*}^a}\sm\{0\}$ such that $D^\a_zD^\b_u(F(z,u)-{\rm Id})=o(u^{\frac{N-1-|\a|-2|\b|}{2}})$ as $(z,u)(\in \wh{M_{r^*}^0}\setminus\{0\})\ra 0$. 
In particular, 
$\wh{M_{r^*}^a}$ is $C^{[\frac{N-1}{2}]}$-smooth. 
\end{proposition}
\begin{proof}
As in the construction of $F_0$, 
restricted to each $B(rp_l, 3r\d)$, $F_0$  is given by  (\ref{deltar}) for a certain initial value function $\xi(t)$ with $D^{\a}_t(\xi(t)-t)=o(r^{N-2})$.
Since $D^{\a}(\Gamma_l-\h{Id})=o(r^{N-2})$,
when restricted to $B(rp_l, 3r\d)$,  we conclude by (\ref{bg}) that $F_0(z)=z+G_l(z)$ where 
\begin{equation}\label{f0gl}
D_z^\a G_l=o(r^{N-1-|\a|})\ \  \h{for}\,\, |\a|\le N-2.
\end{equation}

Next for $Z\in \d_{r^{-1}}(\ov {\O^0_r})$ with $|Z|>c_0$, where $c_0$ is a certain constant such that $B(0,4c_0)\subset\joinrel\subset \d_{r^{-1}}(\O_r^0)$ for  $0<r<r^*$, we define $\eta(Z,r)\ge 1$ such that
$\eta(Z,r)\tilde z\in \d_{r^{-1}}(\p\O_r^0).$ Then $\eta(Z,r)$ is uniquely solved by the following equation:
$$\eta^2(Q(Z,\ov{Z})+(\eta r )^{-2}P(r\eta Z, r\eta\ov{Z}))=1.$$
Hence $\eta(Z,r)$ is
$C^{N_0}$-smooth in its variables, and 
\begin{equation}\label{deta}
D^\a_{Z}D^\b_r\eta(Z,r)=O(r^{-|\b|})\ \  \h{for}\,\, |\a|+|\b|\le N-2. 
\end{equation} 
Let $\chi(Z)\in C_0^{\infty}(\CC^n)$ be such that $0\le \chi(Z)\le 1$, $\chi(Z)=1$ for $3c_0<|Z|<2$ and $\chi(Z)=0$ for $|Z|<2c_0$. 
Now for $z\in \O^0_{\sqrt{u}}$ where $|z|>2c_0\sqrt{u}$,  define
$$F^{0\sharp}(z,u)=\left( (\eta(\frac{z}{\sqrt{u}},\sqrt{u}))^{-1}\cdot F_0(\eta(\frac{z}{\sqrt{u}},\sqrt{u})z),\,\,u\right).$$
Finally, we define
\begin{equation}\label{FF}
 F(z,u)=F^{0\sharp}(z,u) \chi(\frac{z}{\sqrt{u}})+(1-\chi(\frac{z}{\sqrt{u}}))\h{Id}.
\end{equation}
By the construction, it is clear  that $F(z,u)$ is a $C^{N-2}$-smooth map from  $\wh{M^0_{r^*}}\sm\{0\}$ into  $\wh{M^a_{r^*}}\sm \{0\}$ and
$F|_{\p \wh{M^0_{r^*}}\sm\{0\}}=(F_0(z),u)$ is a $C^{N-2}$-diffeomorphism from 
$\p \wh{M^0_{r^*}}\sm \{0\}$ to $\p \wh{M^a_{r^*}}\sm \{0\}$. 

We next prove other statements in the proposition. Let $(z,u)$ be such that $\eta(z,u)z\in B(rp_l,3r\d)$. Then $c^{-1}r\leq\sqrt{u}\leq cr$ for a certain constant $c>1$. Since $F_0(z)=z+G_l(z)$ and (\ref{f0gl}) holds in $B(rp_l,3r\d)$, we have by  (\ref{deta}) that 
$$D^\a_{z}D^{\b}_u \left(F_0(\eta(\frac{z}{\sqrt{u}},\sqrt{u})z)-\eta(\frac{z}{\sqrt{u}},\sqrt{u})z\right)=o(u^{\frac{N-1-|\a|-2|\b|}{2}}).$$
By (\ref{FF}), we conclude that
$D^\a_{z}D^{\b}_u(F(z,u)-\h{Id})=o(u^{\frac{N-1-|a|-2|\b|}{2}}).$
This completes the proof. \end{proof}

Note that the graph function $\Gamma(z)=\left(z,w(z,\ov{z})\right)$ restricted to $S_r\subset\mathbb C^n$ extends to a biholomorphic map from   $\O_r$ to its image, a complex submanifold enclosed by the CR orbit $M_r:=\Gamma(S_r)$ in $M$ which is denoted by  $\wh{M_r}$.   Write $\wh{M^b_{r^*}}:=\cup _{0<r<r^*}\wh{M_r}\cup\{0\}$.  We then  prove  that $\wh{M^b_{r^*}}$ is a $C^{[\frac{N}{2}]-2}$-smooth Levi-flat hypersurface with $M$ near $0$ as part of its smooth boundary. 

\begin{proof}[Proof of Theorem \ref{maintheorem-local}]
Write $(\wt{z},{u})=F(z,u)$ 
and its inverse $(z,u):=(\wt{z},u)+\wt{G}(\wt{z},u)$. 
Then
$D^\a_{\wt{z}}D^{\b}_u(\wt{G}(\wt{z},u))=o(u^{\frac{N-1-|\a|-2|\b|}{2}}).$ 
Define $\wh{w}(\wt{z},\ov{\wt{z}},{u}):=-u+w(\wt{z},\ov{\wt{z}})$. Then $S_{\sqrt{u}}$ is defined by
$$\wt{\rho}(\wt{z},{u}):=-u+(Q+P)(z(\wt{z},{u}),\ov{z(\wt{z},{u})})=0,$$
and for each fixed $u$, $\wh{w}$ is a CR function when restricted to $S_{\sqrt{u}}$. 

Set $0<r^*\ll1$. Similar to the construction in the proof of the classical  Bochner  extension theorem (see Theorem 2.3.2$\empty^{\prime}$ of \cite{Ho}),  we will construct, for   $1\leq j\leq[\frac{N}{2}]-1$,  $C^{N-2-j}$-smooth functions $h_j(\wt{z},\ov{\wt{z}},u)$  on $\{(\wt{z},u):  \wt{z}\in \ov{\Omega_ {\sqrt{u}}},\ 0< u<r^{*2}\}$ 
such that 
$\ov{\p}_{\wt z}w^\sharp=O(\wt{\rho}^{[\frac{N}{2}]-1})$ and for $|\a|+|\b|\le [\frac{N-1}{2}]$, \begin{equation}\label{wsha}
D_{\wt{z}}^\a D^{\b}_{u} w^\sharp=o(u^{\frac{N-|\a|-2|\b|}{2}})\ \ \h{as} \,\,(\wt z,u)(\in \wh{M}^a_{r^*}\setminus\{0\})\ra 0,  
\end{equation}  
where $w^{\sharp}$ is a $C^{[\frac{N-1}{2}]}$-smooth function defined by
\begin{equation}\label{dbar-01}
w^{\sharp}:=\wh{w}(\wt{z},\ov{\wt{z}}, u)
-\wt{\rho}-\sum\nolimits_{j=1}^{[\frac{N}{2}]-1}j^{-1} h_j \wt{\rho}^j.
\end{equation}



For $1\leq k\leq n$, let $\chi_k(Z)\in C_0^{\infty}(V_k)$ be such that $0\le \chi_k(Z)\le 1$ and  $\sum_{k=1}^n\chi_k=\chi$, where $V_k$ are the open cones defined by (\ref{cone}) and $\chi$ is the cutoff function used in (\ref{FF}). Define 
\begin{equation}\label{h1}
h_1:=\sum_{k=1}^n \chi_k(\frac{\wt z}{\sqrt{u}})\cdot\frac{\partial(\wh w-\wt\rho)}{\partial\ov {\wt {z_k}}}(\frac{\partial\wt\rho}{\partial\ov {\wt {z_k}}})^{-1},\,\,g_1=\frac{\ov\partial(\wh w-\wt\rho)-h_1\ov\partial\wt\rho}{\wt\rho}. 
\end{equation}
Here $h_1$ is well-defined for $\frac{\partial\wt\rho}{\partial\ov {\wt {z_k}}}\neq0$ on the support of $\chi_{k}$, and $g_1$ is well-defined for $\frac{\partial\wt\rho}{\partial\ov {\wt {z_j}}}\frac{\partial(\wh w-\wt\rho)}{\partial\ov {\wt {z_k}}}=\frac{\partial\wt\rho}{\partial\ov {\wt {z_k}}}\frac{\partial(\wh w-\wt\rho)}{\partial\ov {\wt {z_j}}}$ on $\{\widetilde\rho=0\}$. Since 
\begin{equation*}
(-\ov\partial h_1+g_1)\wedge\ov\partial\wt\rho=\ov\partial g_1\cdot\wt\rho,    
\end{equation*} 
we have $\frac{\partial\wt\rho}{\partial\ov{\wt  {z_j}}}(-\ov\partial h_1+g_1)_{\ov k}=\frac{\partial\wt\rho}{\partial\ov{\wt  {z_k}}}(-\ov\partial h_1+g_1)_{\ov j}$ on $\{\widetilde\rho=0\}$, where we write $(\mathcal W)_{\ov k}=\lambda_{\ov k}$ for a one form $\mathcal W=\sum{\lambda_{\ov k}}\cdot d\ov {z_k}$. We then define
\begin{equation*}
h_2:=\sum_{k=1}^n\chi_k(\frac{\wt z}{\sqrt{u}})\cdot(-\ov\partial h_1+g_1)_{\ov k}(\frac{\partial\wt\rho}{\partial\ov{\wt  {z_k}}})^{-1},\,\,g_2=\frac{-\ov\partial h_1+g_1-h_2\ov\partial\wt\rho}{\wt\rho}.
\end{equation*}
Notice that
\begin{equation*}
\ov\partial(\wh w-\wt\rho-h_1\wt\rho-\frac{1}{2}h_2\wt\rho^2)=-\wt\rho\ov\partial h_1+\wt\rho g_1-\wt\rho(h_2\ov\partial\rho)-\frac{1}{2}\ov\partial h_2\wt\rho^2=(g_2-\frac{1}{2}\ov\partial h_2)\wt\rho^2=O(\wt\rho^2).
\end{equation*}
Inductively, we define for $2\leq j\leq[\frac{N}{2}]-1$, 
\begin{equation}\label{hgj}
h_j:=\sum_{k=1}^n\chi_k(\frac{\wt z}{\sqrt{u}})\cdot(-\frac{1}{j-1}\ov\partial h_{j-1}+g_{j-1})_{\ov k}(\frac{\partial\wt\rho}{\partial\ov{\wt {z_k}}})^{-1},\,\,g_j=\frac{-\frac{1}{j-1}\ov\partial h_{j-1}+g_{j-1}-h_j\ov\partial\wt\rho}{\wt\rho}. 
\end{equation}
By construction, it is clear that $\ov{\p}_{\wt z}w^\sharp=O(\wt{\rho}^{[\frac{N}{2}]-1})$. On the other hand, computation yields that
\begin{equation*}
D^{\a}_{\wt z}D^{\b}_{u}(\wh w(\wt{z},u)-\wt\rho(\wt{z},u))=o(u^{\frac{N-|\a|-2|\b|}{2}}), D^{\a}_{\wt z}D^{\b}_{u}(\frac{\partial\wt\rho}{\partial\ov{\wt {z_k}}})^{-1}=O(u^{\frac{-1-|\a|-2|\b|}{2}}),  D^{\a}_{\wt z}D^{\b}_{u}(\wt{\rho})=O(u^{\frac{2-|\a|-2|\b|}{2}}). 
\end{equation*}
Then $D^{\a}_{\wt z}D^{\b}_{u}(h_1\wt{\rho})=o({u}^{\frac{N-|\a|-2|\b|}{2}})$. We claim that for $2\leq j\leq[\frac{N}{2}]-1$,
\begin{equation}\label{hj}
D^{\a}_{\wt z}D^{\b}_{u}(h_j\wt{\rho}^j)=o({u}^{\frac{N-|\a|-2|\b|}{2}}),\ \ D^{\a}_{\wt z}D^{\b}_{u}\left((-\frac{1}{j-1}\ov\partial h_{j-1}+g_{j-1})\wt{\rho}^j\right)=o(u^{\frac{N+1-|\a|-2|\b|}{2}}).
\end{equation}
We shall prove by induction. By (\ref{h1}), we have 
\begin{equation*}
\begin{split}
&D^{\a}_{\wt z}D^{\b}_{u}\left((-\ov\partial h_{1}+g_{1})\wt{\rho}^2\right)=D^{\a}_{\wt z}D^{\b}_{u}\left(\wt{\rho}\cdot\ov\partial(\wh w-\wt\rho)-\wt{\rho}\cdot\ov\partial(h_1\wt\rho)\right)=o(u^{\frac{N+1-|\a|-2|\b|}{2}}). 
\end{split}
\end{equation*}
Then (\ref{hj}) holds for $j=2$. Assume that (\ref{hj}) holds for all $j$ such that $2\leq j\leq j^*$ with $2\leq j^*<[\frac{N}{2}]-1$. We proceed to show (\ref{hj}) for $j=j^*+1$. Exploiting (\ref{hgj}), we derive that
\begin{equation*}
\begin{split}
&D^{\a}_{\wt z}D^{\b}_{u}\left((-\frac{1}{j^*}\ov\partial h_{j^*}+g_{j^*})\wt{\rho}^{j^*+1}\right)=D^{\a}_{\wt z}D^{\b}_{u}\left((-\frac{1}{j^*-1}\ov\partial h_{j^*-1}+g_{j^*-1})\wt{\rho}^{j^*}-\wt{\rho}\cdot\ov\partial(\frac{1}{j^*}h_{j^*}\wt{\rho}^{j^*})\right). 
\end{split}
\end{equation*}
Then by induction hypothesis, we conclude the claim. 
We thus have (\ref{wsha}).


Define $C^{[\frac{N}{2}]-2}$-smooth functions $\varphi_j$ by $\varphi_j :=\frac{\partial w^\sharp}{ \partial \ov { \wt{z_j}}}$
on $\{(\wt{z},u):  \wt{z}\in \ov{\Omega_ {\sqrt{u}}},\ 0< u\ll1\}$ and $0$ otherwise. As in the case of solving $\ov\partial$-equations with compact support, letting
$$\Phi(\wt{z},u)=\frac{1}{2\pi i}\int_{\CC}\frac{\varphi_1(\xi,\wt{z_2},\cdots,\wt{z_n})}{\xi-\wt{z_1}}d\xi\wedge \ov{d\xi}=\frac{-1}{\pi}\int_{0}^{2\pi}\int_{0}^{\infty}\varphi_1(\wt{z_1}+\rho e^{i\theta},\wt{z_2},\cdots,\wt{z_n})e^{-i\theta} d\rho d\theta,$$
then $\ov\partial _{\wt z}\Phi =\sum_{j=1}^{n} \varphi_j d\ov{\wt{z_j}}.$
Moreover, we see that $\wt{w^{\sharp}}(\wt{z},u):=w^\sharp(\wt{z},u)-\Phi(\wt{z},u)$
is the holomorphic extension of $\wh{w}|_{S_{\sqrt{u}}}$ to $\O_{\sqrt{u}}$ that is  $C^{[\frac{N}{2}]-2}$-smooth over  $\wh{M_{r^*}}\sm\{0\}$ with 
\begin{equation}\label{aesti}
D_{\wt z}^{\a}D_{{u}}^\b{\wt{w^\sharp}}=o({u}^{\frac{N-1-|\a|-2|\b|}{2}})\ \ \h{for}\,\,  |\a|+|\b|\le[\frac{N}{2}]-2. 
\end{equation}
Now $\mathcal {F}(\widetilde z,u):=\left(\widetilde z,u+\wt{w^\sharp}(\wt z,u)\right)$ 
is a $C^{[\frac{N}{2}]-2}$-smooth diffeomorphism from $\wh{M^a_{r^*}}\sm\{0\}$  to  $\wh{M_{r^*}}\sm\{0\}.$  
Apparently by (\ref{aesti}), we have
$$D_{\wt z}^\a D^\b_u{\mathcal F}\ra 0\ \ \h{as} \,\,(\wt z,u)(\in \wh{M}^a_{r^*}\setminus\{0\})\ra 0\ \ \h{for}\,\,  |\a|+|\b|\le[\frac{N}{2}]-2.$$ Hence $\wh{M^b_{r^*}}$  is of  $C^{[\frac{N}{2}]-2}$-smoothness class by Proposition \ref{mamodel}.  When $N\ge 10$, $\wh{M^b_{r^*}}$  is $C^3$-smooth up to the singular point $0$.  In this case,  we see that   $\wh{M^b_{r^*}}$ is a local holomorphic hull of $M$ near $0$  by applying 
an argument similar to that in  \cite{Hu1}. (See also the later proof of Theorem \ref{hull}.)

When $N\ge 6$, we find a strongly convex domain with $M$ near $0$  as part of its boundary and then regard $\wh{M^b_{r^*}}$ as the graph function over this strongly convex domain. We then have  at least a $C^1$-smooth solution (which is $C^3$-smooth away from the CR  singularity) of the Levi equation. Therefore it is holomorphically convex by a result in \cite{ST}. 

This completes the proof of Theorem \ref{maintheorem-local}. 
\end{proof}
\section{Openness: Local deformation of a smooth Stein hypersurface}

In this section, we provide a proof of Theorem \ref{openness}. 

\begin{proof}[Proof of Theorem \ref{openness}]
Let $p_0\in S_0$. In a local holomorphic chart, we assume $\mathcal M$  near $p_0$ is an open subset of $\CC^{n+1}$ with coordinates $(z,w)=(z_1,\cdots,z_n,w)\in \CC^n\times\CC$. After a holomorphic change of coordinates,  assume $p_0=0$ and ${\rm Span}(T_{0}{S_0}\cup J(T_{0} {S_0}))$ is defined by $\{w=0\}$. 
Define $\pi$ the projection map  $\CC^{n+1}\ra \CC^n$ by sending
$(z,w)$ to $z$. Then, when restricted to each $S_t$  near $p_0$, $\pi$ is a CR diffeomorphism.  Define $\Phi$ from $S$ near $p_0=0$ to the space $(-1,1)\times\CC^{n}$ by sending each $p\in S_t$ to $(t, \pi(p))$. Letting $\psi$ be as in the theorem, notice that $\psi_{*}(\frac{\p}{\p t})$ is transversal to $S_0$ at $p$ by the  hypothesis. On the other hand, let $\rho$ be a strongly plurisubharmonic defining function of $\O$ near $p_0$ and let $\xi(\tau)$ be a   holomoprhic embedding from the unit disk  $\D\subset \CC$ into $\wh{S_0}$ with its image near $p_0$,  which is $C^1$-smooth up to $\ov{\D}$  and is attached to $S_0$ in the sense that  $\xi(\p \D)\subset \p \O$. Assume that $\xi(1)=p_0$.  
Applying the Hopf lemma to $\rho\circ \xi$ at $\tau=1$, $\xi_*$ maps the normal direction of $\D$ at $1$ to a normal direction of $\O$ at $p_0$.   
We  conclude that $\wh{S_0}$ is transversal to $\p \O$.
Then $\Phi$ is a  $C^{N_0-2}$-diffeomorphism from $U_0$, a small neighborhood  of $0$ in $S$, to a certain real hypersurface  of $(-1,1)\times\CC^{n}$  near $(0,0)$. Moreover, $\Phi(t,\cdot)$ is a local CR diffeomorphism. For each fixed $t$, denote by $\Psi(t,\cdot)$ the inverse CR diffeomorphism of $\Phi(t,\cdot)$.  

We write $S^*_t=\Phi(S_t\cap U_0)$, which is a strongly pseudoconvex hypersurface near $\pi(p(t))$. Here $p(t):=\psi(t, p_0)\in S_t$  is a $C^{N_0-2}$-smooth curve transversal to $S_0$. Let $D^*_{t}$  be a small domain near $\pi(p(t))$ in the strongly pseudoconvex side of $S^*_t$ such that any point in  $D^*_{t}$  can be connected by a continuous family of embedded holomorphic disks attached to $S^*_t$ shrinking down to $\pi(p(t))$ and $D^*_{t}$ contains every point in the pseudoconvex side of $S^*_t$ that has Euclidean distance to $\pi(p(t))$ less than a certain positive constant $\e_0\ll1$.  
Since each $S^*_t$ is a $C^{N_0-2}$-smooth family of strongly pseudoconvex manifolds, we  see the $C^{N_0-2}$-smooth dependence of $\Psi(t,\cdot)$  in $t$.

\begin{lemma}\label{ope} 
$\Psi$ extends to a $C^{[N_0/2]-2}$-smooth function over a certain small domain $\ov{\O^*}$ in $(-1,1)\times \CC^n$ such that $\Psi(t,\cdot)$ is holomorphic over  $\O^*\cap D_t^*$, and that $\ov{\O^*\cap D_t^*}$ contains every point  in  the pseudoconvex side of ${S^*_t}$ with a Euclidean distance to $\pi(p(t))$ bounded by a certain $\e_0^{\prime}$ where  $0<\e_0^{\prime}<\e_0\ll1$. 
\end{lemma}
The case for $N_0=\omega$ is a consequence of the classical Cartan theorem (see, for instance, \cite{HY3}). Hence we assume that $N_0\leq\infty$.

\begin{proof}[Proof of Lemma \ref{ope}] We will follow the proof of the classical Lewy extension theorem (see, for instance, \cite{CS}) and study   the dependence on $t$. 

By shrinking $U_0$, the small neighborhood  of $0$ in $S$ as above, and making a  holomorphic  change in $z\in \CC^n$ that depends $C^{N_0-2}$-smoothly on $t$, 
we may assume that the image $\Phi(U_0)$ is a hypersurface defined by a  $C^{N_0-2}$-smooth function 
\begin{equation}\label{boundary}
r(t,z):=-\h{Re}\,z_n+\sum\nolimits_{j=1}^n |z_j|^2+r^*(z,t)=0 
\end{equation}
in a certain domain $\widehat U$ in  $(-1,1)\times\CC^{n}$  containing $(0,\pi(p_0))=(0,0)$. Here $r^*(z,t)=O(|z|^3)$, and $(t, p(t)$) is mapped to $(t,0)$ in the new coordinates.  
Thus, we can  choose $\delta>0$, and then $\epsilon>0$ such that $\frac{\partial^2r}{\partial z_1\overline\partial z_1}>0$ and $\frac{\partial r}{\partial z_n}\neq 0$ on 
\begin{equation}\label{unbound}
U:=\{(t,z)\in(-\e,\e)\times\mathbb C^n: |z_1|<\delta\,\,{\rm and}\,\, |z_2|+|z_3|+\cdots+ |z_n|< \epsilon\}\subset \widehat U    
\end{equation}
and that $r(t,z)>0$ on the part of the boundary of $U$ where $|z_1|=\delta$. Write $U_{-}$ for the subset of $U$ where $r(t,z)<0$.

Now let $f(t,z)$ be a $C^{N_0-2}$-smooth function defined on $U$ such that for each fixed $t$, $f(t,\cdot)$ is a CR function in $z\in S^*_t$. Fix an $N\in \NN$ such that $t\le N\le N_0$, and when $N_0<\infty$,  we set $N=N_0$. As we did in \eqref{dbar-01}, we can find $C^{[N/2]-1}$-smooth functions $h_1,\cdots, h_{[N/2]-1}$ in $(t,z)$ near $(0,0)$ such that $\ov\partial_z\widetilde f(t,z)=O(r^{[N/2]-1})$, where
$$\widetilde f(t,z):=f(t,z)-\sum\nolimits_{j=1}^{[N/2]-1} h_j(t,z)\cdot (r(t,z))^j.$$

Note that for any fixed $(t,z_2,\cdots,z_n)$ in
$$V:=\{(t,z^{\prime}):=(t,z_2,\cdots,z_n)\in(-\e,\e)\times\mathbb C^{n-1}: |z_2|+|z_3|+\cdots+ |z_n|< \epsilon\},$$ 
the set of all $z_1$ with $|z_1|<\delta$ such that $r(t,z)>0$ is connected. We define a $C^{[N/2]-2}$-smooth one form $g:=\sum_{j=1}^n g_j(t,z)d\overline z_j$ on $\CC \times V$ by
$g_j(t,z):=\frac{\partial \widetilde f(t,z)}{\partial\overline z_j}$ 
on $\overline{U_-}$ and $0$ otherwise, for $1\leq j\leq n$.
For any $(t,z_2,z_3\cdots,z_n)\in V$, we define
\begin{equation*}
 G(t,z_1,z_2,\cdots,z_n):=\frac{1}{2\pi i}\int_{\mathbb C}\frac{g_1(t,\xi,z_2,\cdots,z_n)}{\xi-z_1}\,d\xi\wedge d\overline \xi.  
\end{equation*}
Then $\ov\partial G=g$ and $G(t,z)\in C^{[N/2]-2}(\mathbb C\times V)$. 
Now the function $F=\widetilde f-G$ 
is in $C^{[N/2]-2}(U)$ with $F=\widetilde f = f$ on $U\cap \Phi(U_0)$, and $F(t,\cdot)$ is holomorphic in $z$ on $U_-$. 
By choosing $f$ to be the components of $\Psi$, the proof of Lemma \ref{ope} is complete. \end{proof}

We still write $\Psi$ for its holomorphic extension.  Repeat the above procedure at sufficiently many  points  $\{p_j\}_{j=0}^m\subset  S_0$ such that the constructed $\{\ov{\O_j}:=\Psi_j(\ov{\O^*_j})\}_{j=0}^m$ covers a neighborhood of  $S_0$ in $S$.  Then there are  a certain constant $0<\e\ll1$ and a $C^{[N_0/2]-2}$-smooth  submanifold of real  codimension one $\O_\e\subset \mathcal M $ such that the following holds.
\begin{enumerate}[label=(\roman*)]
\item $\O_\e$ is a union of a family of mutually disjointed  complex submanifolds of codimension one; 

\item $\O_\e$ has  
$S$ near $S_0$ as  part of the $C^{[N_0/2]-2}$-smooth boundary;  

\item $\O_{\e}$ stays in the pseudoconvex side of $S$;  

\item 
if $x\in\cup_{t\in (-\e,\e)} \widehat S_t$ is in the pseudoconvex side of $S$ with  the distance of $x$ to $S_0$ is less than $\e$ for a certain fixed metric, then $x\in \O_{\e}$.
\end{enumerate}

Notice that $\widehat S_0$  has no singularity by the assumption. Let $\rho$ be a $C^{[N_0/2]-2}$-smooth defining function of $S$ near $S_0$ in $\ov\O_{\e}$, that is  strongly plurisubharmonic when restricted to $\wh{S_t}$  for $|t|\ll1$. 
For any $0<\e_1\ll\e$, we define Stein spaces
\begin{equation}\label{noboundary}
\wh{S}_{t,\e_1}:=\{z\in \wh{S}_t:\rho(z)<-\e_1\,\ \h{or} \ z\not \in \O_\e\}.
\end{equation}
By a result of Stein (see \cite[Theorem 3.3.3]{Fo}), there is a holomorphic retraction $R$ from a neighborhood $V$ of $\wh{S}_{0}\sm S_0$ in $\mathcal M$
to  $\wh{S}_{0}\sm S_0$ with the retraction along the complex normal direction of  $\wh{S}_{0}\sm S_0$. Also at any point $z_0\in \wh{S_0}\sm S_0$, $R$ can be written in a certain local chart as the projection to its first $n$-components. It is clear that for $0<|t|\ll\epsilon_1\ll\e$, the restriction $R(t,\cdot)$ of $R$ to $\wh{S}_{t, \e_1}$ maps $V\cap \wh{S}_{t, \e_1}\cap \O_{\e}$ biholomorphically to a complex open submanifold $E_{t,\e_1}$ of $\wh{S}_{0}$. 
Note that we can choose $0< \e_1 < \e_2\ll \e$, such that $\p \wh{S}_{0, \e_2}\subset E_{t,\e_1}$ for any $|t|<\e'$ with $0<\epsilon'\ll\e_1$. These regions are depicted in \Cref{fig:egg}.

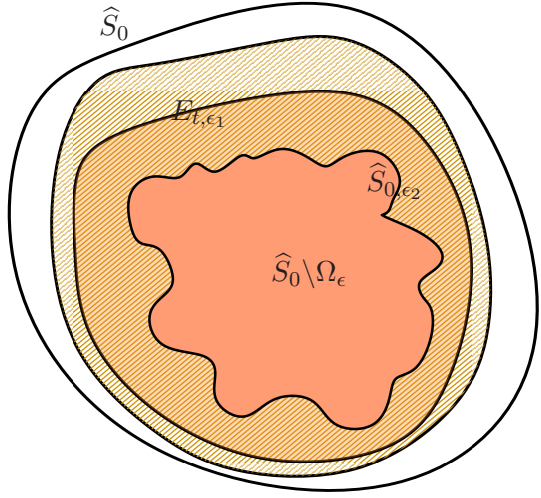
\begin{figure}[htb] {\small
\begin{center}
\begin{overpic}[ scale=0.4]
{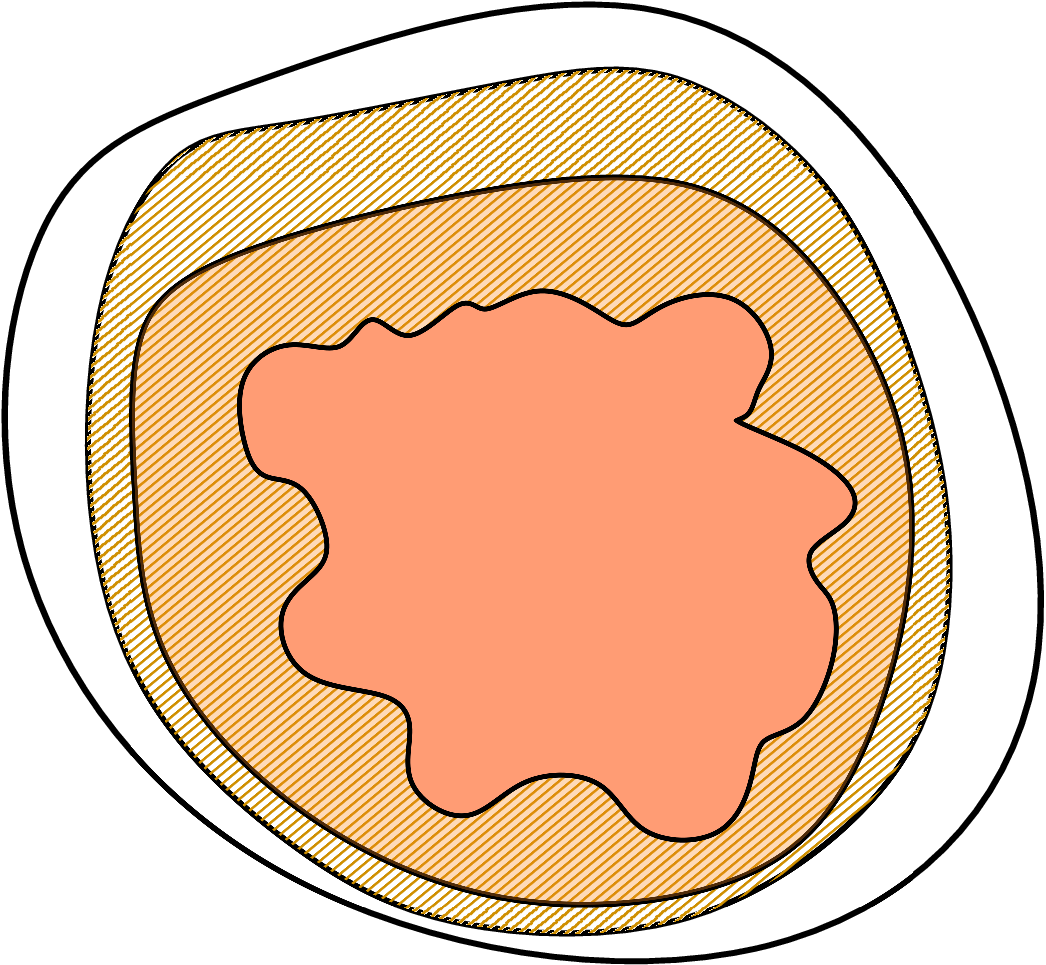}
\put (35,173) {$\wh{S}_0$}
\put (62,141) {$E_{t,\e_1}$}
\put (136,114) {$\wh{S}_{0,\e_2}$}
\put (100,80) {$\wh{S}_0\backslash \O_\e$}
\end{overpic}
\end{center}}
\caption{Relations between different regions.}
\label{fig:egg}
\end{figure}

Embed $\mathcal M$ as a closed complex submanifold in $\CC^m$. We  apply the Hartogs extension theorem to the inversion $(R(t,\cdot))^{-1}$ to get a map $H(t,z):=(h_1(t,z),\cdots, h_{m}(t,z))$ on $(-\e',\e')\times \wh{S}_{0,\epsilon_2}$ such that the following holds:
\begin{enumerate}[label=(\roman*)]
\item $h_i(t,z)$, $1\leq i\leq m$, is holomorphic in $z$ for $(t,z)\in(-\e',\e')\times\wh{S}_{0,\epsilon_2}$, and $H(t,z)$ is of full rank near  $\partial\wh{S}_{0,\epsilon_2}$;
\item $h_i(t,z)$, $1\leq i\leq m$, is $C^{[N_0/2]-2}$-smooth in $(t,z)$ for $(t,z)$ in 
\begin{equation}\label{bd}
\{(t,z):-\e'<t<\e',\,\,z\in (\wh{S}_{0,\e_2}\cap\Omega_{\e})\cup E_{t,\epsilon_1}\}.
\end{equation}
\end{enumerate}


\begin{lemma}\label{hsm}
$H(t,z)$ is a $C^{[N_0/2]-2}$-smooth map on $(-\e',\e')\times \wh{S}_{0,\epsilon_2}$. 
\end{lemma}    
\begin{proof}[Proof of Lemma \ref{hsm}]
In what follows, we shall prove by induction on $k\le [N_0/2]-2$ that $h_i(t,z)$ is 
$C^k$-smooth in $(t,z)$ for all $(t,z)\in(-\e',\e')\times\wh{S}_{0,\epsilon_2}$ and $1\leq i\leq m$.

By the Bochner-Martinelli formula (see, for instance, \cite[Theorem 4.3.3]{HEGE}),  for any holomorphic function $f(z)$ on $\wh{S}_{0,\epsilon_2}$ that is continuous on $\ov {\wh{S}_{0,\epsilon_2}}$, we have that
\begin{equation}\label{bm}
f(z)=\int_{\xi\in\partial \wh{S}_{0,\epsilon_2}}f(\xi)\cdot\psi(z,\xi)\cdot\Omega(z,\xi),\,\,\, \forall\,z\in\wh{S}_{0,\epsilon_2}.
\end{equation}
Here $\Omega(z,\xi)$ is a certain $C^{\infty}$-smooth section in the pull-back of the bundle $\wedge^{2n-1} T^*(\wh{S}_{0,\e_1})$ with respect to the map $\wh{S}_{0,\e_1}\times \wh{S}_{0,\e_1}\ni(z,\xi)\mapsto \xi\in \wh{S}_{0,\e_1}$; 
$\psi(z,\xi)$ is a certain at least $C^2$-smooth function on $(\wh{S}_{0,\e_1}\times \wh{S}_{0,\e_1})\setminus\{(z,z)|z \in \wh{S}_{0,\e_1}\}$ (see \cite[Lemma 4.2.4]{HEGE}).

Setting $f=h_i(t,z)$ in (\ref{bm}), it is clear that $h_i(t,z)$ is continuous on $(-\e',\e')\times \wh{S}_{0,\epsilon_2}$. Taking its $t$-derivative, we derive that 
\begin{equation}\label{cau}
\frac{\partial h_i(t,z)}{\partial t}=\int_{\xi\in\partial \wh{S}_{0,\epsilon_2}}\frac{\partial h_i(t,\xi)}{\partial t}\cdot\psi(z,\xi)\cdot\Omega(z,\xi),\,\,\,\forall\,(t,z)\in(-\e',\e')\times\wh{S}_{0,\epsilon_2}.
\end{equation}
Hence $\frac{\partial h_i(t,z)}{\partial t}$ is continuous on $(-\e',\e')\times \wh{S}_{0,\epsilon_2}$ by 
property (ii).  

For any $z\in\wh{S}_{0,\epsilon_2}$, in a local holomorphic chart, take a polydisc neighborhood $B(\eta_1,r)\times\cdots\times B(\eta_n,r)$ of $z$ in $\wh{S}_{0,\epsilon_2}$ with coordinates $z=(z_1,\cdots,z_n)$. Since $h_i(t,z)$ is holomorphic in $z$, the Cauchy integral formula yields that 
\begin{equation}\label{hol-001}
\begin{split}
h_i(t,z)=\frac{1}{2\pi i}\int_{|\xi_n-\eta_n|=r}\cdots\int_{|\xi_n-\eta_n|=r} \frac{h_i(t,\xi_1,\cdots,\xi_n)}{(\xi_1-z_1)\cdots(\xi_n-z_n)} d\xi_1\cdots d\xi_n,\,\,\,\forall\,t\in(-\e',\e').   
\end{split}    
\end{equation}
It is clear that $\frac{\partial h_i(t,z)}{\partial z_j}$ is continuous in $(t,z)$. Hence,  $h(t,z)$ is  $C^1$-smooth in $(t,z)$. 

Suppose that $h_i(t,z)$ is $C^{k-1}$-smooth in $(t,z)$ for  $(t,z)\in(-\e',\e')\times\wh{S}_{0,\epsilon_2}$. We proceed to show that $h_i(t,z)$ is $C^{k}$-smooth for  $(t,z)\in(-\e',\e')\times\wh{S}_{0,\epsilon_2}$. 

Differentiating (\ref{hol-001}) with respect to $t$, we first conclude that $\frac{\partial h_i(t,z)}{\partial t}$ is holomorphic in $z$. 
It is clear that for any holomorphic vector field $\mathcal D$ on $\wh{S}_{0,\epsilon_1}$, $\mathcal D(h_i)(t,z)$ is holomorphic in $z$. By setting $f=\mathcal D(h_i)(t,z)$ in (\ref{bm}), we conclude by property (ii) that $\mathcal D(h_i)(t,z)$ is continuous on $(-\e',\e')\times \wh{S}_{0,\epsilon_2}$. 
Since any $z$-derivative at $(t,z)\in(-\e',\e')\times \wh{S}_{0,\epsilon_1}$ is generated by a holomorphic vector field on $\wh{S}_{0}$ by Cartan's Theorem A,  it suffices to prove that $\mathcal D_{k}(\mathcal D_{k-1}\cdots(\mathcal D_1(h_i)))$
is continuous on $(-\e',\e')\times \wh{S}_{0,\epsilon_2}$, where each $\mathcal D_j$, $1\leq j\leq k$, is either $\frac{\partial}{\partial t}$ or a holomorphic vector field defined on $\wh{S}_{0,\e_1}$. By induction, $\mathcal D_{k-1}\cdots(\mathcal D_1(h_i))$ is holomorphic in $z$ for $(t,z)\in(-\e',\e')\times\wh{S}_{0,\epsilon_2}$, and $C^{[N_0/2]-k-1}$-smooth in $(t,z)$ for $(t,z)$ in (\ref{bd}). By the same argument as above, we  see that $\mathcal D_{k}(\mathcal D_{k-1}\cdots(\mathcal D_1(h_i)))$ is continuous on $(-\e',\e')\times \wh{S}_{0,\epsilon_2}$. 

The proof of Lemma \ref{hsm} is complete.\end{proof}

Next we proceed to prove that $\widehat S$ is a $C^{[N_0/2]-2}$-smooth real hypersurface near $\widehat S_0$. 

We claim that the Jacobian matrix of map $H(t,z)$ is of full rank, when $t=0$ and $z\in\wh{S}_{0,\epsilon_2}$. Otherwise, assume that $H(t,z)$ is degenerate at $(0,z^*)$ with $z^*\in\wh{S}_{0,\epsilon_2}$. Take local holomorphic coordinates $(z_1,\cdots,z_{n+1})$ in a neighborhood of $z^*\in\mathcal M$ such that $z^*=(z_1,\cdots,z_n,0)$ and $$H(0,z_1,z_2,\cdots,z_n,0)=(\wt{h}_1(0,z_1,z_2,\cdots,z_n,0),\cdots, \wt{h}_{n+1}(0,z_1,z_2,\cdots,z_n,0))\equiv(z_1,\cdots,z_n,0).$$  Here each $\wt{h}_j$ is one of the functions $h_1,\cdots,h_m$.
Then computation yields that $H$ is degenerate at $(0,z^*)$ if and only if $\frac{\partial \wt{h}_{n+1}(t,z)}{\partial t}|_{(0,z^*)}=0$. Since $\frac{\partial \wt{h}_{n+1}(t,z)}{\partial t}|_{t=0}$ is holomorphic in $z$ as shown above, we conclude that the degenerate locus $\mathcal J:=\{z\in\wh{S}_{0,\e_2}:H{\rm\,\, is\,\, degenerate\,\, at\,\,}(0,z)\}$
is a nonempty complex analytic subvariety of $\widehat S_{0,\e_2}$ of codimension one. Thus $\mathcal J$ must intersect $\partial \wh{S}_{0,\e_2}$, for $\dim_{\mathbb C}\mathcal J\geq1$ and $\wh S_{0,\e_2}$ is Stein,  which is a contradiction.


Therefore, there exists $0 <\epsilon \ll 1$ such that for $|t|<\e$, $\widehat S_{\e_2}=\cup_{t\in (-\e,\e)} \widehat S_{t,\e_2}$ is the image of a $C^{[N_0/2]-2}$-smooth embedding and thus is a real hypersurface with the same smoothness. Also the function which assigns points in $\wh{S}_{t,\e_2}$ to $t$ is the $t$ component of $H^{-1}$ and thus is $C^{[N_0/2]-2}$-smooth on $\wh{S}_{\e_2}$ with $dt|_{\wh{S}_{\e_2}}\not =0.$ Similar statements  hold on $\O_\e$ due to the construction of $\O_\e$.
This completes the proof of Theorem \ref{openness} when $N_0\le \infty$.
\smallskip

When $N_0=\omega$,
all $h_i(t,z)$ are real analytic in $(t,z)$ on (\ref{bd}). To prove that $H(t,z)$ is real analytic on $(-\e^{\prime},\e^{\prime})\times \wh{S}_{0,\epsilon_2}$, we can complexify $h_i(t,z)$ and then apply the Hartogs extension on 
\begin{equation*}
\{(\widehat t,z):|\widehat t|<\e^{\prime},\,\,z\in \wh{S}_{0,\epsilon_2}\cap\Omega_{\epsilon}\}.
 \end{equation*}


The proof of Theorem \ref{openness} is complete. \end{proof}

Applying the Baouendi-Treves approximation theorem, one  would  be able  to improve the up-do-boundary  regularity of $\wh{S}$ to $C^{N_0-3}$-smoothness. For simplicity of exposition, we will be content here with the $C^{[N_0/2]-2}$-smoothness of $S$, which    is what we need  for the proof of Theorem \ref{maintheorem-global}   as we do lose about half of the regularity near  CR singularities. (See Remark \ref{optimal-001}.)

\section{Closeness:\  Smoothness  of a complex analytic variety  with a standard contact sphere as its boundary}
 
The result in the previous section serves as the stability part in the proof of Theorem \ref{maintheorem-global}. We provide in this section another key tool when we deform the constructed Stein submanifolds away from elliptic CR singular points in our proof of Theorem \ref{maintheorem-global}.   Our argument here is mainly along the lines of symplectic geometry. The main purpose is to prove  Theorem \ref{nosing}, which may find applications in other related problems.

Recall that a co-oriented contact manifold consists of an odd-dimensional manifold $Y^{2n-1}$ and a codimension one subbundle $\xi\subset TY$ such that there exists a one-form $\alpha$ satisfying $\xi=\ker \alpha$, $\alpha \wedge (d \alpha)^{n-1}\ne 0$, and $\alpha$ gives the co-orientation on $TY/\xi$. 
In particular, $\alpha \wedge (d \alpha)^{n-1}$ gives an orientation of $Y$.  A (co-oriented) contactomorphism is a diffeomorphism preserving the contact structure as well as the co-orientation. 

It is clear that the boundary of a strongly pseudoconvex domain has a contact structure that is naturally induced from its CR structure. 
More precisely, let $\rho$ be a smooth strongly plurisubharmonic function on a complex manifold $(W,J)$. Then any regular level set of $\rho$ is equipped with the contact structure $\xi:=T\rho^{-1}(R)\cap J T\rho^{-1}(R) = \ker(-d^{c}\rho|_{\rho^{-1}(R)})$, where $d^{c}\rho(X):=d\rho(JX)$ for $X\in TW$.  In particular, if $W=\mathbb C^n$ and $\rho=|z|^2$, we get a family of contact manifolds each of which is contactomorphic to the standard contact sphere $(S^{2n-1},\xi_{std})$.

\begin{definition}
Given two smooth contact manifolds $(Y_-,\xi_-)$ and $(Y_+,\xi_+)$, a Liouville cobordism $(W,\lambda)$ from $(Y_-,\xi_-)$ to $(Y_+,\xi_+)$ consists of an oriented smooth  manifold $W$ with boundary  and a Liouville form $\lambda\in \Omega^1(W)$ such that the following holds.
    \begin{enumerate}[label={\rm (\roman*)}]
        \item $d \lambda$ is a symplectic form on $W$;
        \item $\partial W = (-Y_-)\sqcup Y_+$ as oriented manifolds;
        \item the Liouville vector field $X$, defined by $\iota_X d \lambda = \lambda$, points outward along $Y_+$ and points inward along $Y_-$;
        \item $\ker (\lambda|_{Y_{\pm}}) = \xi_{\pm}$.
    \end{enumerate}
    A Liouville cobordism from $\emptyset$ to $Y$ is called a Liouville filling of $Y$.
\end{definition}
\begin{example}\label{level}
Let $\rho$ be a smooth strongly plurisubharmonic function on a complex manifold $(W,J)$. Let $R_1<R_2$ be two regular values of $\rho$ such that $\rho^{-1}([R_1,R_2])$ is compact. Then $\rho^{-1}([R_1,R_2])$ is a Liouville cobordism from the contact manifold $\rho^{-1}(R_1)$ to the contact manifold $\rho^{-1}(R_2)$ where the Liouville form is given by $-d^{c}\rho$. The Liouville vector field is the gradient vector field of $\rho$ with respect to the metric $-d d^{c}\rho(\cdot, J\cdot)$.
\end{example}

The notion of Liouville cobordism is flexible enough to glue cobordisms as follows.
\begin{lemma}\label{prop:glue}
Let $(U,\lambda_U)$ and $(V,\lambda_V)$ be  Liouville cobordisms from $(Y_-,\xi_-)$ to $(Y_+,\xi_+)$ and from $(Z_-,\eta_-)$ to $(Z_+,\eta_+)$, respectively. Assume that there is a contactomorphism $\phi:(Y_+,\xi_+) \to (Z_-,\eta_-)$. Then there exists a Liouville cobordism $(W,\lambda)$ from $(Y_-,\xi_-)$ to $(Z_+,\eta_+)$ such that $W$ has the homotopy type of 
$$U\cup_{\phi} V := U\sqcup V/\{x\simeq \phi(x),\,\,\forall\, x\in Y_+\}.$$
\end{lemma}
\begin{proof}[Proof of Lemma \ref{prop:glue}]
Denote by $X_U$, $X_V$ the Liouville vector fields on $U$, $V$, respectively. Since $X_U$ (resp.~$X_V$) is transversal to $Y_+$ (resp.~$Z_-$),  flowing along $X_U$ (resp.~$X_V$)  yields a diffeomorphism
$$\Phi_U: [-\epsilon,0]\times Y_+ \longrightarrow \overline U_{\epsilon}\subset U \quad ({\rm resp.\,\,}\Phi_V: [0,\epsilon]\times Z_- \longrightarrow \overline V_{\epsilon}\subset V) $$
such that
\begin{equation}
\left\{\begin{aligned}
\frac{\partial\Phi_U(t,x)}{\partial t}&=X_U(\Phi_U(t,x))\\
\Phi_U(0,x)&=x
\end{aligned}\right.\quad \left({\rm resp.\,\,}\left\{\begin{aligned}\frac{\partial \Phi_V(t,x)}{\partial t}&=X_V(\Phi_V(t,x))\\
\Phi_V(0,x)&=x\end{aligned}\right.\right).    
\end{equation}
Here $U_{\epsilon}$, $V_{\epsilon}$ are certain neighborhoods of $Y_+$, $Z_-$  in $U$, $V$, respectively.

We write $\alpha_{Y_+} = \lambda_U|_{Y_+}$ and $\alpha_{Z_-} = \lambda_V|_{Z_-}$. Since $\iota_{X_U}\lambda_U=\iota_{X_U}(\iota_{X_U}d \lambda_U)=0$, pulling back  under $\Phi_U$, we have  $\iota_{\frac{\partial}{\partial t}}(\Phi_U^*(\lambda_U))=0$. Thanks to the Cartan formula, 
\begin{equation}\label{cf}
\mathcal L_{X_U}\lambda_U = \iota_{X_U} d \lambda_U+ d(\iota_{X_U}\lambda_U)=\lambda_U,    
\end{equation}
where $\mathcal L$ is the Lie derivative.   We thus derive that $\Phi_U^*(\lambda_U) = e^t\alpha_{Y_+}$.
Similarly,  $\Phi_V^*(\lambda_V) = e^t\alpha_{Z_-}$. 

If $\phi^*\alpha_{Z_-}=\alpha_{Y_+}$, we define the cobordism by $W:=U\cup_{\phi} V$.   It is clear that the following local coordinate charts define a smooth structure on $W=U\cup_{\phi} V$:  
\begin{enumerate}[label={\rm (\roman*)}]
\item local coordinate charts of $U$ (resp.~$V$) that do not intersect $Y_+$ (resp.~$Z_-$), are considered as local coordinate charts of $W$;

\item the map $(-\epsilon,\epsilon)\times Y_+ \longrightarrow U\cup_{\phi} V$ defined by
\begin{equation*}
\left\{\begin{aligned}
(t,x)&\mapsto \Phi_U(t,x),\,\,\,\,\,\,\,\,\,\,\,\, t\le 0\\
(t,x)&\mapsto \Phi_V(t,\phi(x)),\,\,\,\,\,t\ge 0
\end{aligned}\right.    
\end{equation*}
is declared to be smooth, that is, the image of local coordinate charts of $(-\epsilon,\epsilon)\times Y_+ $ are considered as local coordinate charts of $W$.
    \end{enumerate}
Then it is easy to verify that the one form $\lambda$ on $W$ given by
\begin{equation*}
\lambda|_U=\lambda_U,\,\,\lambda|_V=\lambda_V,\,\,\lambda|_{(-\epsilon,\epsilon)\times Y_+}=e^t\alpha_{Y_+}.    
\end{equation*}
is smooth. In particular, $(W,\lambda)$ defines a Liouville cobordism from $(Y_-,\xi_-)$ to $(Z_+,\eta_+)$.

In general, we have $\phi^*\alpha_{Z_-}= e^f\alpha_{Y_+}$ where $f:Y_+\to \RR$ is a smooth function. Changing $\lambda_U$ to $a\lambda_U$ for $0<a\ll 1$, we may assume that $f>0$. Let $((-\infty,+\infty)\times Y_+, d(e^t\alpha_{Y_+}))$ be the symplectization of $(Y_+,\ker \alpha_{Y_+})$, where $e^t\alpha_{Y_+}$ is a Liouville form with the Liouville vector $\frac{\partial}{\partial t}$. Note that the open submanifold $\{(t,x):0< t < f(x)\}$ of $((-\infty,+\infty)\times Y_+, d(e^t\alpha_{Y_+}))$ is a Liouville cobordism from $(Y_+,\ker \alpha_{Y_+})$ to $(Y_+,\ker (e^f\alpha_{Y_+}))$. We glue this topologically trivial cobordism with $U$ and $V$ at the two ends as above, for the one forms match now. This yields a Liouville cobordism from $(Y_-,\xi_-)$ to $(Z_+,\eta_+)$ which has the homotopy type of $U\cup_{\phi} V$. \end{proof}

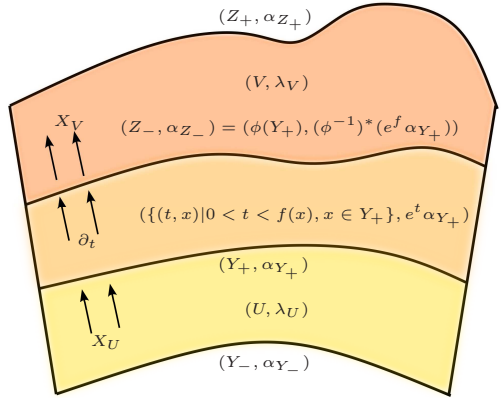
\begin{figure}[htb] {\small
\begin{center}
\begin{overpic}[scale=0.4]
{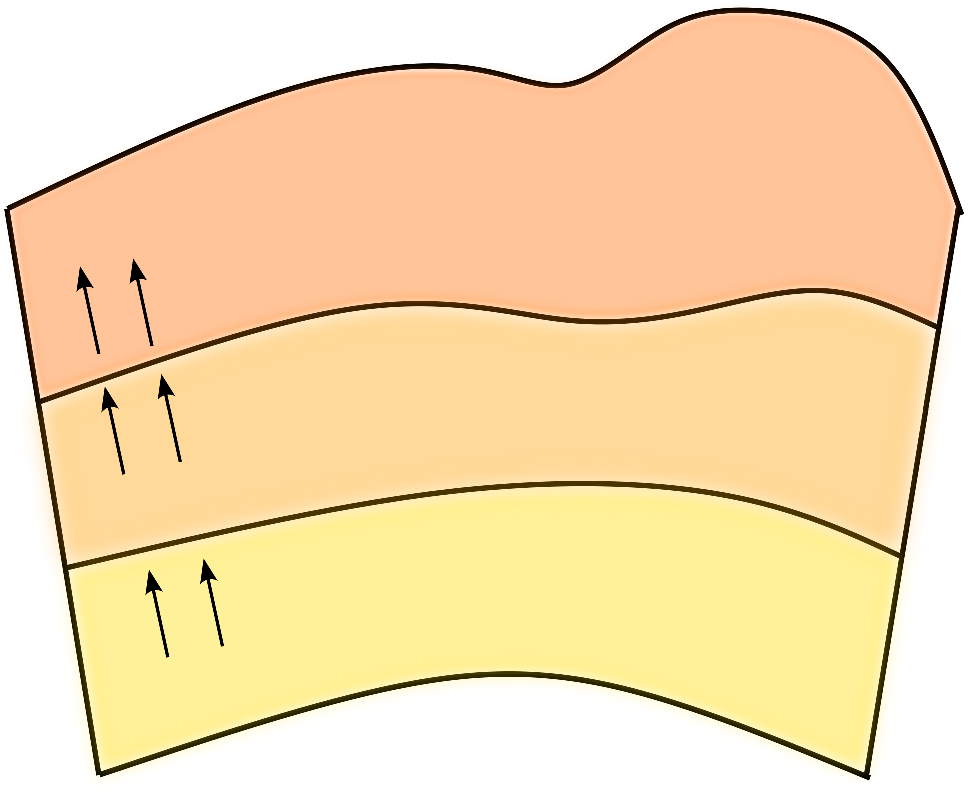}
\put (90,34) {\tiny $(U,\lambda_U)$}
\put (80,51) {\tiny $(Y_+,\alpha_{Y_+})$}
\put (90,120) {\tiny $(V,\lambda_V)$}
\put (80,14) {\tiny $(Y_-,\alpha_{Y_-})$}
\put (50,70) {\tiny $(\{(t,x)|0<t<f(x),x\in Y_+\}, e^t\alpha_{Y_+})$}
\put (80,145) {\tiny $(Z_+,\alpha_{Z_+})$}
\put (43,103) {\tiny $(Z_-,\alpha_{Z_-})=(\phi(Y_+),(\phi^{-1})^*(e^f\alpha_{Y_+}))$}
\put (27,60) {\tiny $\partial_t$}
\put (18,105) {\tiny $X_V$}
\put (32,22) {\tiny $X_U$}
\end{overpic}
\end{center}}
\caption{Gluing of Liouville cobordisms}
\label{fig:Gluing}
\end{figure}

Let $f$ be a holomorphic function defined in a neighborhood 
$U\subset\mathbb C^n$ of  $p$. Assume that $p$ is an isolated singularity of $Z_f:=\{z\in U|f(z)=0\}$, which then is normal when $n\geq 3$ by Serre's criterion.  Define the Milnor link of  $p$ by
\begin{equation}\label{link}
L_{f,p}:=Z_f\cap S(p,\epsilon):=Z_f\cap\{z\in\mathbb C^n:|z-p|^2=\epsilon^2\}, 
\end{equation}
where $S(p,\epsilon)$ is a sufficiently small Euclidean sphere centered at $p$ with radius ${\epsilon}$. For $0<\epsilon\ll1$, $L_{f,p}$ is a compact, strongly pseudoconvex, real-analytic hypersurface of $Z_f$ (see \cite[Propostion 2.4]{looi}),  and thus has a contact structure naturally induced from its CR structure. For simplicity, we call $L_{f,p}$ the contact link of $p$. Note that the standard contact sphere $(S^{2n-1},\xi_{std})$ can be viewed as the contact link of a smooth point. 

Contact links are independent of the choice of $\epsilon$. 
Moreover, the following lemma shows that they can be defined by more general  distance functions.
\begin{lemma}\label{dist} 
Let $\rho_1:U\subset\mathbb C^n\rightarrow [0,\infty)$ be a real-analytic strongly plurisubharmonic function such that $\rho_1^{-1}(0)=\{p\}$ and $\rho_1^{-1}([0,\epsilon_1])$ is compact for a certain  $\epsilon_1>0$. 
Then,  $Z_f\cap \rho_1^{-1}({\epsilon}^2)$ is a compact contact manifold associated with its CR structure,  for any sufficiently small positive number $\epsilon$. Moreover $Z_f\cap \rho_1^{-1}({\epsilon}^2)$ is contactomorphic to $Z_f\cap S(p,\epsilon)$.
\end{lemma}
\begin{proof}[Proof of Lemma \ref{dist}]
It is clear that  $Z_f\cap \rho_1^{-1}({\epsilon}^2)$ is a compact contact manifold for $0<\epsilon\ll1$. 

In what follows, we shall connect $Z_f\cap \rho_1^{-1}({\epsilon}^2)$ with $Z_f\cap S(p,\epsilon)$ by a smooth family of contact manifolds. Write $\rho_0=|z-p|^2$. We define a family of real-analytic strongly plurisubharmonic functions on $U$ by $\rho_t:=(1-t)\rho_0+t\rho_1$ for $0\leq t\leq 1$. Then there is an 
$\tilde\epsilon>0$ such that $Z_f\cap \rho_t^{-1}(\epsilon^2)$ is compact for $0\leq t\leq 1$, $0<\epsilon<\tilde\epsilon$.
Moreover, by  the Curve Selection Lemma, there is a neighborhood $V$ of $p$ in $Z_f$ such that at no point of $V\setminus\{p\}$, $d(\rho_0|_{V\setminus\{p\}})$ and $d(\rho_1|_{V\setminus\{p\}})$ point in opposite directions, that is, if $d(\rho_0|_{V\setminus\{p\}})$ and $d(\rho_1|_{V\setminus\{p\}})$ are proportional, then the
factor of proportionality is never negative
(see \cite[Propostion 2.5]{looi}). Then there is a  positive constant $\hat\epsilon$ such that $Z_f\cap \rho_t^{-1}(\epsilon^2)$ is a compact smooth hypersuface of $Z_f$ for $0\leq t\leq 1$, $0<\epsilon<\hat\epsilon$.



Now by the Gray stability theorem \cite{Gr2}, we conclude that  $Z_f\cap \rho_1^{-1}(\epsilon)$ is contactomorphic to $Z_f\cap S(p,\epsilon)$ for $0<\epsilon\ll1$. \end{proof}

\begin{proposition}\label{prop:Milnor}  Let $L_{f,p}$ be the contact link of a non-smooth isolated hypersurface singularity $p$  as defined by (\ref{link}). Then there is a Liouville filling of $L_{f,p}$ that has the homotopy type of a bouquet $\vee_k S^n$ with $k\ge 1$.
\end{proposition}
\begin{proof}[Proof of Proposition \ref{prop:Milnor}]
By \cite[Proposition 6.19]{dim},  there is a germ of biholomorphisms $g$ of $(\mathbb C^{n},p)$ such that $f\circ g$ is the germ of a polynomial.  Define germs of real analytic functions $\rho_0:=|z-p|^2$ and $\rho_1:=\rho_0\circ g$. Then the contact link $L_{f,p}=Z_f\cap S(p,\epsilon)$ is contactomorphic to the contact manifold $g^{-1}(L_{f,p})=Z_{f\circ g}\cap \rho_1^{-1}(\epsilon^2)$  for $0<\epsilon\ll1$. By Lemma \ref{dist}, we further conclude that $g^{-1}(L_{f,p})$ is contactomorphic to the contact link $L_{f\circ g,p}=Z_{f\circ g}\cap S(p,\epsilon)$ for $0<\epsilon\ll1$.  Hence, we may assume that $f$ is a polynomial in the following. 

Fix $0<\epsilon \ll 1$. By  \cite[Theorem 2.8]{looi}, $f^{-1}(B(0,\delta))\cap S(p,\epsilon)$ is a smooth fiber bundle over $B(0,\delta)$ for $0<\delta \ll \epsilon $, where $B(0,\delta):=\{z\in\mathbb C:|z|<\delta\}$, and, moreover, each fiber is a contact manifold. Since 
$L_{f,p}$ is contactomorphic
to $f^{-1}(b)\cap S(p,\epsilon)$ by the Gray stability theorem,  by Example \ref{level} $L_{f,p}$ has a Liouville filling  that is homotopic to the Milnor fiber 
$$\mathcal X_{f,p}:
=f^{-1}(b)\cap \{z\in\mathbb C^n:|z-p|^2<\epsilon^2\},$$
where $b\in B(0,\delta)\setminus\{0\}$.
By \cite[Theorems 6.5 and 7.1]{mil2}, $\mathcal X_{f,p}$ has the homotopy type of a bouquet $\vee_k S^n$ with $k\geq 1$, for $p$ is non-smooth. \end{proof}

The key input from contact topology in this work is the following rigidity result on Liouville fillings for the link of a smooth point, i.e.\ the standard contact sphere. 
\begin{theorem}[Eliashberg-Floer-McDuff]\label{efm}
Any Liouville filling of $(S^{2n-1},\xi_{std})$ is diffeomorphic to the complex unit ball $\mathbb B^{n}$ for $n\ge 2$.
\end{theorem}
The above theorem   was obtained by Gromov \cite[0.3.C.]{Gromov} when $n=2$. The general case was due to Eliashberg-Floer-McDuff \cite[Theorem 1.5]{McDuff}. The proofs in these papers were based on the study of pseudo-holomorphic spheres. An alternative proof, based on the Floer cohomology, can be found in \cite[Theorem 1.2]{zhou} by the fourth author.

\begin{proof}[Proof of Theorem \ref{nosing}]
We first prove that $W$ is smooth. Otherwise, by the assumption, $W$ has finitely many isolated hypersurface singularities $p_1,p_2,\cdots,p_l$ with $l\geq1$.

In what follows, we first construct a special smooth strongly plurisubharmonic Morse defining function $\rho$ of $\overline W$. 
For each $1\leq i\leq l$, take an embedding $\iota_i$ of the germ of the isolated hypersurface singularity $(U_i,p_i)$ into $\mathbb C^{n+1}$ with coordinates $z=(z_1,\cdots,z_{n+1})$. Denote by $h_i$ the pullback under $\iota_i$ of the function $(|z|^2+\log |z|^2)$ defined on $\mathbb C^{n+1}$. Applying the Curve Selection Lemma, after shrinking $U_i$ if needed, we can assume that $h_i$ has no critical points in $U_i\setminus\{p_i\}$. By cutting off $h_i$ around $p_i$ and extending it by zero outside, we can derive a function $\eta$ such that $\eta=h_i$ near $p_i$. Embedding $W$ properly into $\mathbb C^N$ and denoting by $\tilde\eta$ the pullback of the function $|Z|^2$ with $Z\in \CC^N$, we can thus define a strongly plurisubharmonic exhaustion function on $W$ by
$\rho_1:=\exp(\eta+k\cdot\tilde\eta)$ for $k\gg1$. On the other hand, there is a plurisubharmonic smooth defining function $\rho_2$ of $\overline W$ such that \begin{enumerate}[label=(\roman*)]
\item $d\rho_2|_{\partial W}\neq 0$, $\rho_2\equiv 0$ on $\partial W$, and $\rho_2<0$ on $W$, 

\item $\rho_2$ is strongly plurisubharmonic in a certain neighborhood $U$ of $\partial W$,

\item $\rho_2$  takes minimum on a subset containing $\bigcup_{1\leq i\leq l} U_i$.
\end{enumerate}
Now replacing $\rho_1$ by $\rho_1-\mathcal N$ for a sufficiently large $\mathcal N$,  cutting off $\rho_1$ near the boundary, adding $k\cdot\rho_2$ to it with $k\gg1$, and then applying the Morse approximation, we  get the desired  $\rho$. Note that $\rho$ assumes its minimum exactly at  $p_1,p_2,\cdots,p_l$, and, moreover, $\rho$ has only finitely many non-degenerate critical points in $W\setminus\{p_1,p_2,\cdots,p_l\}$.

Without loss of generality, we assume that the $\rho(p_i)=-1$. Take $0<\epsilon\ll1$ such that $(-1+\epsilon)$ is a regular value of $\rho$, and $\rho^{-1}([-1,-1+\epsilon))$ is a disjoint union of germs of isolated hypersurface singularities with contact links $L_{f_1,p_1}$, $L_{f_2,p_2},\cdots,L_{f_l,p_l}$.  It is known that the Morse index of each critical point in $\rho^{-1}([-1+\epsilon,0])$ is at most $n$. Then the standard Morse theory 
yields that $\rho^{-1}([-1+\epsilon,0])$ has the homotopy type of $\rho^{-1}(-1+\epsilon)$ with cells of real dimension $\leq n$ attached.

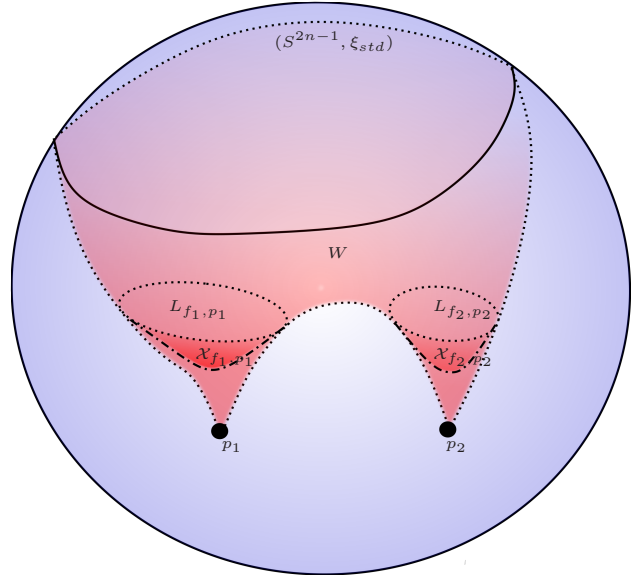
\begin{figure}[htb] {\small
\begin{center}
\begin{overpic}[scale=0.4]
{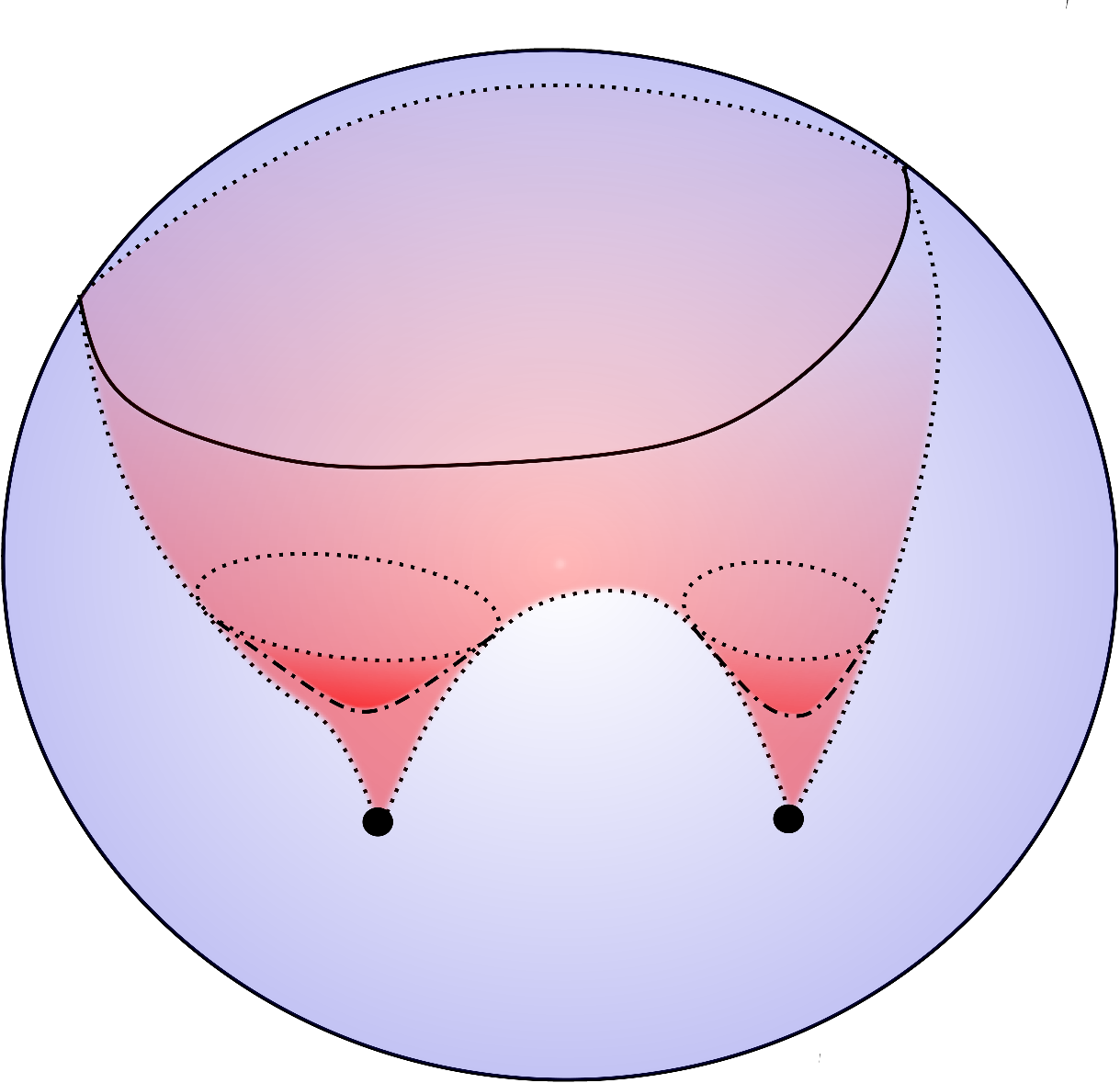}
\put (80,48) {\tiny $p_1$}
\put (165,48) {\tiny $p_2$}
\put (100,200) {\tiny $(S^{2n-1},\xi_{std})$}
\put (120,120) {\tiny $W$}
\put (160,100) {\tiny $L_{f_2,p_2}$}
\put (60,100) {\tiny $L_{f_1,p_1}$}
\put (70,82) {\tiny $\mathcal{X}_{f_1,p_1}$}
\put (160,82) {\tiny $\mathcal{X}_{f_2,p_2}$}
\end{overpic}
\end{center}}
\caption{Schematic picture for replacing the neighborhood by Milnor fibers}
\label{fig:Milnor}
\end{figure}

Now $\rho^{-1}([-1+\epsilon, 0])$ is a Liouville cobordism from the disjoint union of the contact links $L_{f_i,p_i}$ to the standard contact sphere $(S^{2n-1},\xi_{std})$ by assumption
(see Example \ref{level}). On the other hand, we  derive a Liouville filling $V$ of $(S^{2n-1},\xi_{std})$ by gluing the Milnor fibers $\mathcal X_{f_i,p_i}$ of $p_i$ to the Liouville cobordism  $\rho^{-1}([-1+\epsilon,0])$ according to Proposition \ref{prop:Milnor} and Lemma \ref{prop:glue}. Since each ${\mathcal X}_{f_i,p_i}$ has the homotopy type of a bouquet $\vee_{k_i} S^n$ with $k_i\geq 1$, we can thus conclude that $H_n(V)\neq 0$, which contradicts \Cref{efm}.

Therefore $W$ is smooth. By \Cref{efm} again, $W$ is diffeomorphic to ${\BB}^{n}$ as  it is a Liouville filling of $(S^{2n-1},\xi_{std})$. \end{proof}
\begin{remark} \label{rem-002}
Since any $C^1$-smooth deformation connecting two smooth contact manifolds can be made into a smooth deformation, for the space of contact structures is open in the $C^1$-topology, the above proof applies to the case when $\partial W$ is $C^k$-smooth, $k\geq2$, as a slight push-in of $\partial W$, which is smooth and strongly pseudoconvex, is smoothly contactomorphic to $(S^{2n-1},\xi_{std})$.
\end{remark}

Let $(X, p)$ be a germ of isolated singularities with $\dim_{\mathbb C}X=n$. Recall that $(X, p)$ is said to be an isolated complete intersection singularity if it is a complete intersection, that is, it can be embedded in $(\mathbb C^{n+k}, p)$ as
the zero locus of an ideal $I\subset\mathcal O_{\mathbb C^{n+k},p}$ generated by $k$ functions $f_1,\cdots, f_{k}$. In particular, an isolated hypersurface singularity is an isolated complete intersection singularity.

We mention that  the same method as above also gives the following slightly stronger result than our Theorem \ref{nosing}:
\begin{theorem}\label{icis}
Let $W$ be a Stein space of dimension $n\geq 2$ with at most finitely many isolated complete intersection singularities. Assume further that $W$ has a compact smooth strongly pseudoconvex boundary  $\partial W$,  of which the CR structure is contactomorphic to   $(S^{2n-1}, \xi_{std})$. Then $W$ is a smooth Stein manifold and is diffeomorphic to the complex unit ball $\BB^n$.  
\end{theorem}
\begin{proof}[Proof of Theorem \ref{icis}] Let $p$ be a (non-smooth) isolated complete intersection singularity of $W$. Then the germ $(X,p)$ of $W$ at $p$ can be defined by a germ of a holomorphic map $f:=(f_1,\cdots,f_{k}):(\mathbb C^{n+k},p)\rightarrow (\mathbb C^{k},0)$ such that $X=f^{-1}(0)$ and $\frac{\partial f_i}{\partial z_j}(p)=0$ for any $1\leq i\leq k$, $1\leq j\leq n+k$, where $(z_1,\cdots,z_{n+k})$ are coordinates of $\mathbb C^{n+k}$. By \cite{LOS}, the Milnor number $\mu(X,p)$ of $(X,p)$ is no less than its Tjurina number $\tau(X,p)$, which, as in \cite{GLS}, is given by
\begin{equation*}
\tau(X,p):=\dim\frac{\mathcal O^k_{\mathbb C^{n+k},p}}{\mathcal O^{n+k}_{\mathbb C^{n+k},p}\cdot (\frac{\partial f_j}{\partial z_i})_{\substack{1\leq i\leq n+k\\1\leq j\leq k}}+\mathcal O^k_{\mathbb C^{n+k},p}\cdot f_1+\mathcal O^k_{\mathbb C^{n+k},p}\cdot f_2+\cdots+\mathcal O^k_{\mathbb C^{n+k},p}\cdot f_{k}}.
\end{equation*}
It is then clear that $\mu(X,p)\geq 1$.

As in (\ref{link}), we define the Milnor (contact) link of $(X, p)$ to be $L_{f,p}:=f^{-1}(0)\cap S(p,\epsilon)$, where $0<\epsilon\ll1$. 
Write $\rho=|z-p|^2$. By \cite[Theorem 2.8]{looi}, there are a number $0<\delta \ll \epsilon $ and a complete analytic subvariety $D\subsetneq B(0,\delta):=\{z\in\mathbb C^k:|z|<\delta\}$ such that $f^{-1}(B(0,\delta)\setminus D)\cap\rho^{-1}([0,\epsilon^2])$ and $f^{-1}(B(0,\delta)\setminus D)\cap \rho^{-1}(\epsilon^2)$ are smooth fiber bundles over $B(0,\delta)\setminus D$, and that each fiber $f^{-1}(b)\cap \rho^{-1}(\epsilon^2)$ with $b\in B(0,\delta)\setminus D$ is a contact manifold. By \cite[Corollary 5.9]{looi}, each fiber $f^{-1}(b)\cap\rho^{-1}([0,\epsilon^2])$ has
the homotopy type of a finite bouquet of $n$-spheres, where the number of $n$-spheres is the Milnor number $\mu(X,p)\geq1$  of $(X,p)$.

By the same argument as in Theorem  \ref{nosing}, we now complete the proof of Theorem \ref{icis}. \end{proof}

In view of Theorems \ref{nosing}, \ref{icis}, we make the following conjecture. 
\begin{conjecture}\label{conj}
Let $W$ be a Stein space of dimension $n\geq 3$ with at most finitely many isolated normal singularities. Assume further that $W$ has a compact strongly pseudoconvex smooth boundary $\partial W$, of which the CR structure is contactomorphic to  $(S^{2n-1}, \xi_{std})$. Then $W$ is a smooth Stein manifold and is diffeomorphic to the complex unit ball $\BB^n$.
\end{conjecture}

Characterizing the smoothness of isolated singularities from a local perspective has a long history. A fundamental theorem in this direction due to Mumford \cite{Mum} states that a non-smooth normal surface singularity can not have a simply connected link. In higher dimensions, the presence of a sphere link alone is insufficient to rule out isolated singularities as demonstrated by certain Brieskorn singularities. 
Motivated by these considerations, Seidel conjectured that an isolated normal singularity is smooth if and only if its contact link is the standard contact sphere. In \cite{McL},  McLean discovered the deep relation between the Reeb dynamics on the contact link and the minimal discrepancy of the singularity, an invariant of singularities coming up in the minimal model program. In particular, he proved Seidel's conjecture conditionally to a conjecture of Shokurov on minimal log discrepancies, which has been answered affirmatively in
dimension $3$ and for locally complete intersection singularities. 

Conjecture \ref{conj} is a global version of Seidel's conjecture. McLean's relation between the Reeb dynamics on the contact boundary and the minimal discrepancy of the singularities does not  hold 
for general Stein spaces other than a small neighborhood of one singularity, due to possible non-trivial symplectic topology outside germs of singularities. In particular, it is not even clear if Shokurov's conjecture would imply Conjecture \ref{conj}. When $n=2$, Conjecture \ref{conj} is derived readily from the works of Gromov \cite{Gromov}, Eliashberg \cite{Eli89} and McDuff \cite{McDuff90} as any symplectic filling of $(S^3,\xi_{std})$ must be blow-ups of the symplectic ball, hence the resolution of singularities in $W$ as a symplectic filling of $(S^3,\xi_{std})$ must be blow-ups of a ball. In fact, in this scenario it suffices to assume that $\partial W$ is diffeomorphic to $S^3$ by  \cite[Corollary 5.3]{Eli89}. Conjecture \ref{conj} has also been answered affirmatively in  \cite[Theorem C] {gironella2021exact} when the singularities are assumed to be quotient singularities. Conjecture \ref{conj} would be a  powerful tool in tackling many problems in several complex variables, if confirmed affirmatively.

\section {Proof of Theorem \ref{hull} and completion  of the proof of Theorem \ref{maintheorem-global}}

We now assume that  $M\subset {\mathcal M}$ is as in Theorem \ref{maintheorem-global}.  Let $p_1$ and $p_2$ be two elliptic CR singular points of $M$. Other points on $M$ are  CR non-minimal. Near $p_1$ or $p_2$, we may assume that $\mathcal M$ is an open subset of $\CC^{n+1}$.    For two leaves $M_1$ and $M_2$ sufficiently close to two elliptic CR singular points, respectively, by Theorem \ref{maintheorem-local},  $D(M_1)\cup D(M_2)\sm \{p_1,p_2\}$  is foliated by CR orbits of $M$ with its boundaries  $M_1$ and $M_2$ as two leaves. Here 
$D(M_1)$ and $D(M_2)$  are formed by joining  closed orbits from $p_1$ to $M_1$ and from $p_2$ to $M_2$ (but not including  $M_1$ and $M_2$), respectively. Write $\wh{D(M_1)}$ and $\wh{D(M_2)}$ for the unions of the families of Stein submanifolds with boundary enclosed by closed orbits in $D(M_1)\sm\{p_1\}$ and $D(M_2)\sm\{p_2\}$, respectively.
  
By the Reeb-Thurston stability theorem \cite{Th},    there is  a $C^{\ell(N_0)}$-smooth diffeomorphism  $G$ from $[1,2]\times S^{2n-1}$ to $M\sm \{D(M_1)\cup D(M_2)\}$ such that
$G(t,\cdot)$ maps the unit sphere $S^{2n-1}\times\{t\}$ in $\CC^n\times\{t\}$ to a leaf  $M_t$ in $M$. Moreover, $G(S^{2n-1}\times\{1\})=M_1$ and $G(S^{2n-1}\times\{2\})=M_2$. 
By a theorem of Harvey-Lawson \cite{HL}, each $M_t$ bounds a complex hypervarity $\wh{M_t}$ with possible isolated normal singularities. 
Let $\wh{M}=\{p_1,p_2\}\cup {\wh {D(M_1)}}\cup \ {\wh{D(M_2)}}\cup_{1\le t\le 2}\wh{M_t}$. Notice that by Theorem \ref{maintheorem-local}, the Reeb-Thurston stability theorem, and the Gray stability theorem, each $M_t$ is contactmorphic to the unit sphere in $\CC^n$ equipped with the standard contact structure.  Then, by Theorem \ref{maintheorem-local},  Theorem \ref{openness}  and  Theorem \ref{nosing} with Remark \ref{rem-002},  we conclude that  $\wh{M}\setminus\{p_1,p_2\}$ is a $C^{\ell(N_0)} $-smooth Levi-flat hypersurface foliated by smooth complex hypersurfaces that is $C^{\ell(N_0)}$-smooth up to $M$. The uniqueness of $\wh{M}$ follows from the fact that each compact CR orbit bounds a unique Stein submanifold or follows from Theorem  \ref{hull} if $N_0\ge 10$. 

To complete the proof of Theorem \ref{maintheorem-global}, it remains to prove Theorem \ref{hull}. By the continuity principle, it is apparent that any pseudoconvex domain in $\mathcal M$ containing $M$ contains $\wh{M}$. It suffices to construct a family of pseudoconvex domains shrinking down to $\wh{M}$. 

Our construction here    is partially motivated by the earlier work of Huang \cite{Hu1} for the construction  near a singularity and that of  Forstneri\v{c}-Laurent-Thi\'{e}baut \cite{FLT} of the construction of   good defining functions at  a regular point.

Let $\Omega\subset\mathcal M$ be a $C^2$-smoothly bounded strongly pseudoconvex domain of which the boundary contains $M$. Let $\rho$ be a strongly plurisubharmonic defining function of $\Omega$ defined on a certain neighborhood of $\ov\Omega$. 
Inspired by a beautiful idea of Forstneri\v{c}-Laurent-Thi\'{e}baut \cite{FLT}, we first construct a special definition function of $\widehat M$ as follows.
\begin{lemma}\label{regular}
There is an open neighborhood $\mathcal U$ of $\widehat M\cup M$ in $\mathcal M$ and a $C^3$-smooth function $\widehat v$ on $\mathcal U$ such that the following holds: \begin{enumerate}[label=(\roman*)]
    \item $\wh v(x)=0$ and $d\wh v(x)\neq 0$ for all $x\in\widetilde M$, where $\widetilde{M}:=\{x\in {\mathcal U}:\wh v(x)=0\}$ is a certain $C^3$-smooth real hypersurface in $\mathcal U$ such that $\widehat M\subset\joinrel\subset\widetilde {M}$.
    
    \item Denote by ${\rm dist}(x,\widehat M)$ the distance from $x$ to $\wh M$ with respect to a certain fixed metric on $\mathcal M$. Then as $x\in{\mathcal U}$ and ${\rm dist}(x,\widehat M)\rightarrow 0$,
$$|\partial{\ov\partial}\wh v(x)|=o({\rm dist}(x,\widehat M))=o(|\wh v(x)|+|\rho(x)|).$$  
\end{enumerate}
\end{lemma}
\begin{proof}[Proof of Lemma \ref{regular}] Let $M_1$ and $M_2$ be two leaves sufficiently close to $p_1,p_2$, respectively, as above. 
By the Reeb-Thurston stability theorem,  there is a $C^3$-smooth function $u$ on $M\sm \{p_1,p_2\}=\cup_{0<t<3}{M_t}$ such that $du|_{M\sm \{p_1,p_2\}}\neq0$ and $M\sm \{D(M_1)\cup D(M_2)\}\cup\{p_1,p_2\}=\cup_{1\le t \le 2}{M_t}$. By Theorem \ref{openness}, we can define a $C^3$-smooth function $\widehat u$ on $\wh{M}\sm \{p_1,p_2\}=\cup_{0<t<3}\wh{M_t}$ 
such that $\widehat u$ has constant value on each $\wh{M_t}$ and $d\wh{u}|_{\wh{M}\sm \{p_1,p_2\}}\not =0$.

Let $A_0:=\cup_{1-\epsilon_1<t<2+\epsilon_1}\wh{M_t}$ where $0<\epsilon_1\ll1$. From  the proof of Theorem \ref{openness}, any $p\in A_0\setminus  M$ is covered by a local coordinate chart $(V_{\alpha},\phi_{\alpha})$ such that $\phi_{\alpha}$ is at least $C^3$-smooth and 
\begin{equation*}
\phi_{\alpha}(V_{\alpha})=\{(t,z_1,\cdots,z_{n}):|t|<\epsilon,\,\,|z_1|^2+\cdots+|z_{n}|^2<\e\}\subset\mathbb R\times\mathbb C^{n}   
\end{equation*}
with $0<\e\ll1$ (see (\ref{noboundary})). Here points on the same  leaf have the same  $t$-value, and $\phi_{\a}^{-1}$ is holomorphic for fixed $t$. By the Whitney extension theorem, we have a $C^3$-smooth extension $\Phi_{\a}$ of $\phi_{\alpha}^{-1}$ to an open neighborhood 
$$B_{\alpha,\delta}:=\{(z_0=t+iT,z_1,\cdots,z_{n}):|t|<\epsilon,|T|<\delta,|z_1|^2+\cdots+|z_n|^2<\epsilon^2\}\subset\mathbb C^{n+1}$$
with $0<\delta\ll\epsilon$ 
such that 
\begin{equation*}
\ov \partial \Phi_{\beta}=o(T^2)\,\,\,{\rm as}\,\,T\rightarrow0.   
\end{equation*} 
Write $\mathcal V_{\alpha}:=\Phi_{\a}(B_{\alpha,\delta})$.
Similarly, any $p\in M\cap A_0$ is covered by a local coordinate chart $(V_{\beta},\phi_{\beta})$ such that $\phi_{\beta}$ is at least $C^3$-smooth and 
\begin{equation*}
\phi_{\beta}(V_{\beta})=\{(t,z_1,\cdots,z_{n}):r(t,z)\leq0,\,\,|t|<\epsilon,|z_1|^2+\cdots+|z_n|^2<\epsilon^2\}\subset\mathbb R\times\mathbb C^{n}
\end{equation*}
with $r$ given by (\ref{boundary}) and $0<\epsilon\ll1$.  
By the Whitney extension theorem, we have a $C^3$-smooth extension $\Phi_{\beta}$ of $\phi_{\beta}^{-1}$ to an open neighborhood 
$$B_{\beta,\delta}:=\{(z_0=t+iT,z_1,\cdots,z_{n}):r(t,z)<\delta,\,\,|t|<\epsilon,|T|<\delta,|z_1|^2+\cdots+|z_n|^2<\epsilon^2\}\subset\mathbb C^{n+1}$$
with $0<\delta\ll\epsilon$ 
such that 
\begin{equation}\label{dbv}
\ov \partial \Phi_{\beta}(z)=o(T^2)+o(({\rm dist}(z,\phi_{\beta}(V_{\beta})))^2)\,\,\,{\rm as\,\,}{\rm dist}(z,\phi_{\beta}(V_{\beta}))\rightarrow 0,    
\end{equation}
where ${\rm dist}(z,\phi_{\beta}(V_{\beta}))$ is the distance from $z$ to $\phi_{\beta}(V_{\beta})$ with respect to the Euclidean metric. Write $\mathcal V_{\beta}:=\Phi_{\beta}(B_{\beta,\delta})$. 


Notice that $(\phi^{-1}_{\beta})^*(\widehat u)=u_{\beta}(t)$ where $u_{\beta}(t)$ is $C^3$-smooth and $\frac{du_{\beta}(t)}{dt}\neq 0$. By the Whitney extension theorem, there is $C^3$-smooth function $\xi$ on $B_{\beta,\delta}$ such that $\xi|_{\phi_{\beta}(V_{\beta})}=u_{\beta}$ and that $\overline{\partial}\xi$ vanishes to order 
$2$ on $B_{\beta,\delta}\cap\{T=0\}$. 
Since $i\partial\overline{\partial}$ is a real operator, $\partial\overline{\partial}(\h{Im}\,\xi)$ vanishes to order $1$ on $B_{\beta,\delta}\cap\{T=0\}$. 
Define $\widehat v_{\beta}=(\h{Im}\, \xi)\circ \Phi_{\beta}$. Then  $d\widehat v_{\beta}|_{B_{\beta,\delta}\cap\{T=0\}}\neq0$ for $\frac{du_{\beta}(t)}{dt}\neq 0$.
Writing  
$(\wt z_{\wt 0},\cdots, \wt z_{\wt n})$ for the coordinates of the target, computation yields that \begin{equation*}\begin{split}&\frac{\partial^2\widehat v_{\beta}}{\partial z_j\ov{\partial}z_k}=\frac{\partial}{\partial z_j}\sum_{\gamma=0}^n(\frac{\partial\h{Im}\,\xi}{\partial\wt z_{\gamma}}\frac{\partial  \Phi_{\b}^{\wt\gamma}}{\partial\ov{z_k}}+\frac{\partial\h{Im}\,\xi}{\partial\ov{\wt z_{\gamma}}}\frac{\partial\ov{\Phi_{\b}^{ \wt \gamma}}}{\partial\ov{z_k}})=\sum_{\delta,\gamma=0}^n(\frac{\partial^2\h{Im}\,\xi}{\partial\wt z_{\delta}\partial\wt z_{\gamma}}\frac{\partial \Phi_{\b}^{ \wt\delta}}{\partial{z_j}}\frac{\partial \Phi_{\b}^{\wt\gamma}}{\partial\ov{z_k}}+\frac{\partial^2\h{Im}\,\xi}{\partial\ov{\wt z_{\delta}}\partial\wt z_{\gamma}}\frac{\partial\ov{\Phi_{\b}^{\wt\delta}}}{\partial{z_j}}\frac{\partial\Phi_{\b}^{\wt\gamma}}{\partial\ov{z_k}}\\&+\frac{\partial^2\h{Im}\,\xi}{\partial\wt z_{\delta}\partial\ov{\wt {z_{\gamma}}}}\frac{\partial\Phi_{\b}^{\wt\delta}}{\partial{z_j}}\frac{\partial \ov{\Phi_{\b}^{\wt\gamma}}}{\partial\ov{z_k}}+\frac{\partial^2\h{Im}\,\xi}{\partial\ov{\wt z_{\delta}}\partial\ov{\wt {z_{\gamma}}}}\frac{\partial\ov{\Phi_{\b}^{\wt\delta}}}{\partial{z_j}}\frac{\partial \ov{\Phi_{\b}^{\wt\gamma}}}{\partial\ov{z_k}})+\sum_{\gamma=0}^n(\frac{\partial\h{Im}\,\xi}{\partial\wt z_{\gamma}}\frac{\partial^2\Phi_{\b}^{\wt\gamma}}{\partial z_j\partial\ov{z_k}}+\frac{\partial\h{Im}\,\xi}{\partial\ov{\wt z_{\gamma}}}\frac{\partial^2\ov{\Phi_{\b}^{\wt\gamma}}}{\partial z_j\partial\ov{z_k}}).\end{split}   
\end{equation*}
By (\ref{dbv}), $\widehat v_{\beta}$ satisfies properties (i), (ii) locally. In the same way, we can derive the desired defining function $\widehat v_{\alpha}$ locally in a small open neighborhood of $V_{\a}$, by finding approximate harmonic conjugate of $\widehat u$ (similar to that in \cite[Propositions 1.2, 3.1]{FLT}). Moreover, on the overlaps in $\widehat M$, local defining functions agree at least to order $3$ up to the boundary $M$. 

Next we consider the points near CR singularities.
When $M$ is real analytic, it follows by the result of Huang-Yin in \cite{HY3} that, near the elliptic CR singular point $p_1$ or $p_2$, $\widehat M$ can be biholomorphically transformed to a domain defined by $\widehat v:=\h{Im}\, w=0$, $\h{Re}\, w\geq G(z,\ov z)$, and that closed leaves of $\widehat M$ are cut out by $\h{Re}\,w=$ constant, where $(z,w)$ are coordinates of $\mathbb C^{n+1}$.  Then for $i=1,2$, there are open neighborhoods $\mathcal U_i$ of $\ov{D(M_i)}$ in $\mathcal M$, and $C^3$-smooth functions $\widehat v_i$ on $\mathcal U_i$ such that properties (i), (ii) hold. We arrange the defining function of the foliation $\widehat u$ on $A_0=\cup_{1-\epsilon_1<t<2+\epsilon_1}\wh{M_t}$ so that it coincides with the function $\h{Re}\,w$ on $\wh{D(M_i)}\cap A_0$.

Now assume that $M$ is $C^{N_0}$-smooth, where $N_0\geq 10$. Let $p_1,p_2$ be two elliptic CR singular points. 
By Proposition \ref{mamodel} and the proof of Theorem \ref{maintheorem-local}, we have a $C^{3}$-smooth diffeomorphism $F$ from $\wh{M^a_{\epsilon}}:=\{(t,z)\in \RR\times\CC^n:  z\in \ov{\Omega_{\sqrt{t}}},\ 0< t<{\epsilon}\}\cup\{0\}$ to an open neighborhood of $p_1$ (or $p_2$) in $\widehat M$ such that along each leaf $\Omega_{\sqrt{t}}$, $F$ is holomoprhic. Here $\wh{M^a_{\epsilon}}$ is itself a $C^3$-smooth domain  with boundary in $\RR\times \CC^n$. By the Whitney extension theorem,  we  derive a $C^3$-smooth extension $\wh F$ of $F$ to an open neighborhood 
$$B_{\delta}:=\{(z_0=t+iT,z_1,\cdots,z_{n}):|t|<\epsilon,|T|<\delta,|z_1|^2+\cdots+|z_n|^2<\epsilon^2\}\subset\mathbb C^{n+1}$$
with $0<\delta\ll\epsilon$ 
such that 
\begin{equation*}
\ov \partial \widehat{F}(z)=o(T^2)+o(({\rm dist}(z,\wh{M^a_{\epsilon}}))^2)\, \,\,{\rm as\,\,}{\rm dist}(z,\wh{M^a_{\epsilon}})\rightarrow 0,    
\end{equation*}
where ${\rm dist}(z,\wh{M^a_{\epsilon}})$ is the distance from $z$ to $\wh{M^a_{\epsilon}}$ with respect to the Euclidean metric. 
Define $\widehat v=T\circ (\widehat F)^{-1}$. 
By a similar computation as in \cite{Hu1}, 
$\widehat v$ satisfies properties (i), (ii) locally.  Then for $i=1,2$, there are open neighborhoods $\mathcal U_i$ of $\ov{D(M_i)}$ in $\mathcal M$, and $C^3$-smooth functions $\widehat v_i$ on $\mathcal U_i$ such that properties (i), (ii) hold. Similarly,  we arrange the defining function of the foliation $\widehat u$ on $A_0$ so that it coincides with the function $t\circ(\widehat F)^{-1}$ on $\wh{D(M_i)}\cap A_0$.


Take a $C^3$-smooth partition of unity $\{\chi_1,\chi_2,\chi_{\alpha},\cdots,\chi_{\beta},\cdots\}$  subordinate  to a finite cover $\{\mathcal U_1,\mathcal U_2,\mathcal V_{\alpha},$ $\cdots,\mathcal V_{\beta},\cdots\}$.  We now define 
$$\widehat v:=\chi_{1}\wh v_1+\chi_2\wh v_2+\sum\nolimits_{\a}\chi_{\a}\wh v_{\a}+\sum\nolimits_{\b}\chi_{\b}\wh v_{\b}=:\sum\nolimits_{\gamma}\chi_{\gamma}\wh v_{\gamma}.$$ Let $\mathcal U$ be a small neighborhood of $\wh{M}$. On the support of each $\chi_{\gamma_0}$, computation yields that 
\begin{equation*}
d\wh v=\sum\nolimits_{\gamma}(d\chi_{\gamma}\wh v_{\gamma}+ d\wh v_{\gamma}\chi_{\gamma})=\sum\nolimits_{\gamma}(d\chi_{\gamma}(\wh v_{\gamma}-\wh v_{\gamma_0})+ d\wh v_{\gamma}\chi_{\gamma}),    
\end{equation*}
which equals to $d\wh v_{\gamma_0}\neq 0$ when restricted to $\wh M\cap{\rm Supp}\,\chi_{\gamma_0}$. Similarly, on the support of each $\chi_{\gamma_0}$, computation yields that 
\begin{equation*}
\begin{split}
\frac{\partial^2 \wh v}{\partial z_j\partial\ov{z_k}}=&\sum\nolimits_{\gamma}(\frac{\partial^2 \chi_{\gamma}}{\partial z_j\partial\ov{z_k}}\wh v_{\gamma}+\frac{\partial^2 \wh v_{\gamma}}{\partial z_j\partial\ov{z_k}}\chi_{\gamma}+\frac{\partial \wh v_{\gamma}}{\partial z_j}\frac{\partial \chi_{\gamma}}{\partial\ov{z_k}}+\frac{\partial \chi_{\gamma}}{\partial z_j}\frac{\partial \wh v_{\gamma}}{\partial\ov{z_k}})\\
=&\sum\nolimits_{\gamma}(\frac{\partial^2 \chi_{\gamma}}{\partial z_j\partial\ov{z_k}}(\wh v_{\gamma}-\wh v_{\gamma_0})+\frac{\partial^2 \wh v_{\gamma}}{\partial z_j\partial\ov{z_k}}\chi_{\gamma}+\frac{\partial (\wh v_{\gamma}-\wh v_{\gamma_0})}{\partial z_j}\frac{\partial \chi_{\gamma}}{\partial\ov{z_k}}+\frac{\partial \chi_{\gamma}}{\partial z_j}\frac{\partial (\wh v_{\gamma}-\wh v_{\gamma_0})}{\partial\ov{z_k}}).
\end{split}    
\end{equation*}
We therefore see the proof of  Lemma \ref{regular}. \end{proof}

\begin{proof}[Proof of Theorem \ref{hull}]
The proof is the same as that in \cite[Proposition 3.3]{FLT}.  
Choose $\mathcal U$ to be sufficiently close to $\wh{M}$ such that $d\wh{v}\not = 0$ in $\mathcal U$, where $\widehat v$ is the defining function constructed in Lemma \ref{regular}. Let $\widetilde{M}$ be the zero set of $\wh{v}$ in $\mathcal U$.  
Choose a small $c>0$ such that 
$\{p\in\widetilde{M} \colon  \rho(p) \le c\} \subset\joinrel\subset  \mathcal U$
and each value in $[0,c]$ is a regular value of $\rho|_{\widetilde{\mathcal M}}$.
For sufficiently 
small $\epsilon>0$, we set $\widehat v_\epsilon^\pm(p) = \pm \widehat v(p) + \epsilon (\rho(p)-c)$ where $p\in\mathcal U$, and define
\begin{equation*}
\Omega_\epsilon:= \{p\in\mathcal U: \widehat v^+_\epsilon(p)<0,\ \widehat v^{-}_\epsilon(p)<0,\         \rho(p)<\epsilon \}. 
\end{equation*}
It is clear that $\partial\Omega_\epsilon$ consists of three $C^2$-smooth strongly pseudoconvex  hypersurfaces (towards $\O_\epsilon$): $\Gamma_\epsilon^\pm=\{p\in\mathcal U:  \widehat v_\epsilon^\pm(p)=0,\ \rho(p)<\epsilon\}$, $\{p\in\mathcal U: \rho(p)=\epsilon,\ \widehat v^+_\epsilon(p)<0,\ \widehat v^{-}_\epsilon(p)<0\}$, and that $\Omega_\epsilon$ 
shrinks down to $\widehat M$  as $\epsilon \to 0$. By Lemma \ref{regular}, we conclude that $\wh v_\epsilon^\pm$ are strongly plurisubharmonic 
on $\Omega_\epsilon$ provided that $0<\epsilon\ll1$. Hence, $\Omega_\epsilon$ are Stein domains for $0<\epsilon\ll1$.

The proof of Theorem \ref{hull} and thus the proof of Theorem \ref{maintheorem-global} are  complete.  \end{proof}


\bibliographystyle{plain} 
\bibliography{ref}

\begin{thebibliography}{10}

\bibitem{BER}
M.~S. Baouendi, P.~Ebenfeldt, and Linda~Preiss Rothschild.
\newblock Local geometric properties of real submanifolds in complex space.
\newblock {\em Bull. Am. Math. Soc., New Ser.}, 37(3):309--336, 2000.

\bibitem{BG}
Eric Bedford and Bernard Gaveau.
\newblock Envelopes of holomorphy of certain 2-spheres in {{\({\mathbb{C}}^ 2\)}}.
\newblock {\em Amer. J. Math.}, 105:975--1009, 1983.

\bibitem{BK}
Eric Bedford and Wilhelm Klingenberg.
\newblock On the envelope of holomorphy of a 2-sphere in {{\(\mathbb{C}^ 2\)}}.
\newblock {\em J. Am. Math. Soc.}, 4(3):623--646, 1991.

\bibitem{Bis}
Errett Bishop.
\newblock Differentiable manifolds in complex {Euclidean} space.
\newblock {\em Duke Math. J.}, 32:1--21, 1965.

\bibitem{Bur1}
Valentin Burcea.
\newblock A normal form for a real 2-codimensional submanifold in {{\(\mathbb C^{N+1}\)}} near a {CR} singularity.
\newblock {\em Adv. Math.}, 243:262--295, 2013.

\bibitem{Bur2}
Valentin Burcea.
\newblock On a family of analytic discs attached to a real submanifold {{\(M\subset \mathbb{C}^{N+1}\)}}.
\newblock {\em Methods Appl. Anal.}, 20(1):69--78, 2013.

\bibitem{CS}
So-Chin Chen and Mei-Chi Shaw.
\newblock {\em Partial differential equations in several complex variables}, volume~19 of {\em AMS/IP Studies in Advanced Mathematics}.
\newblock American Mathematical Society, Providence, RI; International Press, Boston, MA, 2001.

\bibitem{dim}
Alexandru Dimca.
\newblock {\em Topics on real and complex singularities}.
\newblock Advanced Lectures in Mathematics. Friedr. Vieweg \& Sohn, Braunschweig, 1987.
\newblock An introduction.

\bibitem{Di}
Tien-Cuong Dinh.
\newblock Enveloppe polynomiale d'un compact de longueur finie et cha\^{i}nes holomorphes \`a bord rectifiable.
\newblock {\em Acta Math.}, 180(1):31--67, 1998.

\bibitem{DTZ0}
Pierre Dolbeault, Giuseppe Tomassini, and Dmitri Zaitsev.
\newblock On boundaries of {Levi}-flat hypersurfaces in {{\(\mathbb C^n\)}}.
\newblock {\em C. R., Math., Acad. Sci. Paris}, 341(6):343--348, 2005.

\bibitem{DTZ1}
Pierre Dolbeault, Giuseppe Tomassini, and Dmitri Zaitsev.
\newblock On {Levi}-flat hypersurfaces with prescribed boundary.
\newblock {\em Pure Appl. Math. Q.}, 6(3):725--753, 2010.

\bibitem{DTZ2}
Pierre Dolbeault, Giuseppe Tomassini, and Dmitri Zaitsev.
\newblock Boundary problem for {Levi} flat graphs.
\newblock {\em Indiana Univ. Math. J.}, 60(1):161--170, 2011.

\bibitem{Eli89}
Yakov Eliashberg.
\newblock Filling by holomorphic discs and its applications.
\newblock Geometry of low-dimensional manifolds. 2: {Symplectic} manifolds and {Jones}-{Witten}-{Theory}, {Proc}. {Symp}., {Durham}/{UK} 1989, {Lond}. {Math}. {Soc}. {Lect}. {Note} {Ser}. 151, 45-72 (1990)., 1990.

\bibitem{FH}
Hanlong Fang and Xiaojun Huang.
\newblock Flattening a non-degenerate {CR} singular point of real codimension two.
\newblock {\em Geom. Funct. Anal.}, 28(2):289--333, 2018.

\bibitem{FM}
John~Erik Forn{\ae}ss and Daowei Ma.
\newblock A {$2$}-sphere in {${\bf C}^2$} that cannot be filled in with analytic disks.
\newblock {\em Internat. Math. Res. Notices}, (1):17--22, 1995.

\bibitem{For}
Franc Forstneri\v{c}.
\newblock Complex tangents of real surfaces in complex surfaces.
\newblock {\em Duke Math. J.}, 67(2):353--376, 1992.

\bibitem{Fo}
Franc Forstneri\v{c}.
\newblock {\em Stein manifolds and holomorphic mappings}, volume~56 of {\em Ergebnisse der Mathematik und ihrer Grenzgebiete. 3. Folge. A Series of Modern Surveys in Mathematics}.
\newblock Springer, Heidelberg, 2011.

\bibitem{FLT}
Franc Forstneri\v{c} and Christine Laurent-Thi\'{e}baut.
\newblock Stein compacts in {L}evi-flat hypersurfaces.
\newblock {\em Trans. Amer. Math. Soc.}, 360(1):307--329, 2008.

\bibitem{gironella2021exact}
Fabio Gironella and Zhengyi Zhou.
\newblock Exact orbifold fillings of contact manifolds.
\newblock {\em arXiv preprint arXiv:2108.12247}, 2021.

\bibitem{Gon1}
Xianghong Gong.
\newblock Normal forms of real surfaces under unimodular transformations near elliptic complex tangents.
\newblock {\em Duke Math. J.}, 74(1):145--157, 1994.

\bibitem{Gon2}
Xianghong Gong.
\newblock Existence of real analytic surfaces with hyperbolic complex tangent that are formally but not holomorphically equivalent to quadrics.
\newblock {\em Indiana Univ. Math. J.}, 53(1):83--95, 2004.

\bibitem{GL}
Xianghong Gong and Ji{\v{r}}{\'{\i}} Lebl.
\newblock Normal forms for {CR} singular codimension-two {Levi}-flat submanifolds.
\newblock {\em Pac. J. Math.}, 275(1):115--165, 2015.

\bibitem{GS1}
Xianghong Gong and Laurent Stolovitch.
\newblock Real submanifolds of maximum complex tangent space at a {CR} singular point. {I}.
\newblock {\em Invent. Math.}, 206(2):293--377, 2016.

\bibitem{GS2}
Xianghong Gong and Laurent Stolovitch.
\newblock Real submanifolds of maximum complex tangent space at a {CR} singular point. {II}.
\newblock {\em J. Differ. Geom.}, 112(1):121--198, 2019.

\bibitem{Gr2}
John~W. Gray.
\newblock Some global properties of contact structures.
\newblock {\em Ann. of Math. (2)}, 69:421--450, 1959.

\bibitem{GLS}
G.-M. Greuel, C.~Lossen, and E.~Shustin.
\newblock {\em Introduction to singularities and deformations}.
\newblock Springer Monographs in Mathematics. Springer, Berlin, 2007.

\bibitem{Gromov}
Misha Gromov.
\newblock Pseudo holomorphic curves in symplectic manifolds.
\newblock {\em Invent. Math.}, 82:307--347, 1985.

\bibitem{GW}
Purvi Gupta and Chloe~Urbanski Wawrzyniak.
\newblock Stability of the hull(s) of an {$n$}-sphere in {$\mathbb{C}^n$}.
\newblock {\em Adv. Math.}, 392:Paper No. 107989, 33, 2021.

\bibitem{HL}
F.~Reese Harvey and H.~Blaine Lawson, Jr.
\newblock On boundaries of complex analytic varieties. {I}.
\newblock {\em Ann. Math. (2)}, 102:223--290, 1975.

\bibitem{HEGE}
Gennadi Henkin and J\"{u}rgen Leiterer.
\newblock {\em Theory of functions on complex manifolds}, volume~79 of {\em Monographs in Mathematics}.
\newblock Birkh\"{a}user Verlag, Basel, 1984.

\bibitem{Hofer}
Helmut Hofer.
\newblock Pseudoholomorphic curves in symplectizations with applications to the {Weinstein} conjecture in dimension three.
\newblock {\em Invent. Math.}, 114(3):515--563, 1993.

\bibitem{Ho}
Lars H{\"{o}}rmander.
\newblock {\em An introduction to complex analysis in several variables}, volume~7 of {\em North-Holland Mathematical Library}.
\newblock North-Holland Publishing Co., Amsterdam, third edition, 1990.

\bibitem{Hu1}
Xiaojun Huang.
\newblock On an {{\(n\)}}-manifold in {{\({\mathbf{C}}^{n}\)}} near an elliptic complex tangent.
\newblock {\em J. Am. Math. Soc.}, 11(3):669--692, 1998.

\bibitem{HK}
Xiaojun Huang and Steven~G. Krantz.
\newblock On a problem of {Moser}.
\newblock {\em Duke Math. J.}, 78(1):213--228, 1995.

\bibitem{HY1}
Xiaojun Huang and Wanke Yin.
\newblock A {Bishop} surface with a vanishing {Bishop} invariant.
\newblock {\em Invent. Math.}, 176(3):461--520, 2009.

\bibitem{HY2}
Xiaojun Huang and Wanke Yin.
\newblock A codimension two {CR} singular submanifold that is formally equivalent to a symmetric quadric.
\newblock {\em Int. Math. Res. Not.}, 2009(15):2789--2828, 2009.

\bibitem{HY5}
Xiaojun Huang and Wanke Yin.
\newblock Flattening of {CR} singular points and analyticity of local hull of holomorphy.
\newblock {\em arXiv:1210.5146}, 2012.

\bibitem{HY3}
Xiaojun Huang and Wanke Yin.
\newblock Flattening of {CR} singular points and analyticity of the local hull of holomorphy {I}.
\newblock {\em Math. Ann.}, 365(1-2):381--399, 2016.

\bibitem{HY4}
Xiaojun Huang and Wanke Yin.
\newblock Flattening of {CR} singular points and analyticity of the local hull of holomorphy {II}.
\newblock {\em Adv. Math.}, 308:1009--1073, 2017.

\bibitem{KW1}
Carlos~E. Kenig and Sidney~M. Webster.
\newblock The local hull of holomorphy of a surface in the space of two complex variables.
\newblock {\em Invent. Math.}, 67:1--21, 1982.

\bibitem{KW2}
Carlos~E. Kenig and Sidney~M. Webster.
\newblock On the hull of holomorphy of an n-manifold in {{\({\mathbb{C}}^ n\)}}.
\newblock {\em Ann. Sc. Norm. Super. Pisa, Cl. Sci., IV. Ser.}, 11:261--280, 1984.

\bibitem{KS}
Martin Klime\v{s} and Laurent Stolovitch.
\newblock Reversible parabolic diffeomorphisms of {$(\mathbb{C}^2,0)$} and exceptional hyperbolic {CR}-singularities.
\newblock {\em arXiv:2204.09449}, 2022.

\bibitem{Leb}
Ji{\v{r}}{\'{\i}} Lebl.
\newblock Extension of {Levi}-flat hypersurfaces past {CR} boundaries.
\newblock {\em Indiana Univ. Math. J.}, 57(2):699--713, 2008.

\bibitem{LNR1}
Ji{\v{r}}{\'{\i}} Lebl, Alan Noell, and Sivaguru Ravisankar.
\newblock Codimension two {CR} singular submanifolds and extensions of {CR} functions.
\newblock {\em J. Geom. Anal.}, 27(3):2453--2471, 2017.

\bibitem{looi}
Eduard Looijenga.
\newblock {\em Isolated singular points on complete intersections}, volume~5 of {\em Surveys of Modern Mathematics}.
\newblock International Press, Somerville, MA; Higher Education Press, Beijing, second edition, 2013.

\bibitem{LOS}
Eduard Looijenga and Joseph Steenbrink.
\newblock Milnor number and {T}jurina number of complete intersections.
\newblock {\em Math. Ann.}, 271(1):121--124, 1985.

\bibitem{McDuff90}
Dusa McDuff.
\newblock The structure of rational and ruled symplectic 4-manifolds.
\newblock {\em J. Am. Math. Soc.}, 3(3):679--712, 1990.

\bibitem{McDuff}
Dusa McDuff.
\newblock Symplectic manifolds with contact type boundaries.
\newblock {\em Invent. Math.}, 103(3):651--671, 1991.

\bibitem{McL}
Mark McLean.
\newblock Reeb orbits and the minimal discrepancy of an isolated singularity.
\newblock {\em Invent. Math.}, 204(2):505--594, 2016.

\bibitem{mil2}
John Milnor.
\newblock {\em Singular points of complex hypersurfaces}.
\newblock Annals of Mathematics Studies, No. 61. Princeton University Press, Princeton, NJ; University of Tokyo Press, Tokyo, 1968.

\bibitem{Mos}
J{\"u}rgen~K. Moser.
\newblock Analytic surfaces in {{\({\mathbb{C}}^ 2\)}} and their local hull of holomorphy.
\newblock {\em Ann. Acad. Sci. Fenn., Ser. A I, Math.}, 10:397--410, 1985.

\bibitem{MW}
J{\"u}rgen~K. Moser and Sidney~M. Webster.
\newblock Normal forms for real surfaces in {{\({\mathbb{C}}^ 2\)}} near tangents and hyperbolic surface transformations.
\newblock {\em Acta Math.}, 150:255--296, 1983.

\bibitem{Mum}
David Mumford.
\newblock The topology of normal singularities of an algebraic surface and a criterion for simplicity.
\newblock {\em Inst. Hautes \'{E}tudes Sci. Publ. Math.}, 9:5--22, 1961.

\bibitem{ST}
Zbigniew Slodkowski and Giuseppe Tomassini.
\newblock The {Levi} equation in higher dimensions and relationships to the envelope of holomorphy.
\newblock {\em Amer. J. Math.}, 116(2):479--499, 1994.

\bibitem{SZ}
Laurent Stolovitch and Zhiyan Zhao.
\newblock Geometry of hyperbolic {Cauchy}-{Riemann} singularities and {KAM}-like theory for holomorphic involutions.
\newblock {\em Math. Ann.}, 386(1-2):587--672, 2023.

\bibitem{Th}
William~P. Thurston.
\newblock A generalization of the {R}eeb stability theorem.
\newblock {\em Topology}, 13:347--352, 1974.

\bibitem{zhou}
Zhengyi Zhou.
\newblock On fillings of {{\(\partial (V\times \mathbb{D})\)}}.
\newblock {\em Math. Ann.}, 385(3-4):1493--1520, 2023.

\end{thebibliography}
\Addresses
\end{document}